\documentclass[11pt]{article}
\usepackage{fancyhdr,graphicx}
\usepackage{amsfonts}
\usepackage{amssymb}
\usepackage{amsthm}
\usepackage{epstopdf}
\usepackage{newlfont}
\usepackage{amsmath} %%%%%%%%%
\usepackage[frame,arrow,curve,matrix]{xy}
\usepackage{subfigure}

\newtheorem{thm}{Theorem}[section]

\usepackage{mathrsfs}  %������ѧϣ������$\mathscr{OS}$
\usepackage{subfig}
\usepackage{picinpar}
\begin{document}

\title{Twisted braids}

\author{Shudan Xue \quad Qingying Deng \footnote{Corresponding author. \newline {\em E-mail address:} qingying@xtu.edu.cn (Q.Deng).} \\
{\footnotesize   \em School of Mathematics and Computational Science, Xiangtan University, Xiangtan, Hunan 411105, P. R. China}}

%\date{\today}

\maketitle

\begin{abstract}
Twisted knot theory, introduced by M.O. Bourgoin, is a generalization of virtual knot theory.
It naturally yields the notion of a twisted braid, which is closely related to the notion of a virtual braid due to Kauffman.
In this paper, we first prove that any twisted link can be described as the closure
of a twisted braid, which is unique up to certain basic moves. This is the analogue of the Alexander Theorem and the Markov Theorem for classical braids and links.
Then we also give reduced presentations for the twisted braid group and the flat twisted braid group. These reduced presentations are based on the fact that these
twisted braid groups on $n$ strands are generated by a single braiding element and a single bar element plus the generators of the symmetric group on $n$ letters.
\end{abstract}

$\mathbf{keywords:}$ twisted link; twisted braids; braiding link diagram; reduced presentation.

\footnote{This article will appear in INVOLVE Issue 4.}

\vskip0.5cm

\section{Introduction}
\setlength{\parindent}{2em}
\indent In 1996, Kauffman \cite{LHK} introduced virtual knot theory, which is a generalization of classical knot theory, and mentioned virtual braid is a subject very close to the ``welded braid" of Fenn, Rimanyi and Rourke in \cite{RRC}. Subsequently, Kauffamn and Lambropoulou did further research on virtual braids in \cite{Kauffman2000, Kauffman2004, Kauffman2005}.

In 1923, Alexander \cite{JWA} stated that any classical link is described as the closure of a certain braid, that is, the Alexander Theorem. Other proofs of classical Alexander Theorem refer to \cite{SL1, SL2, SL3, HRM, P.V, SY}.
Note that this kind of braid is not the only one. In 1936, Markov \cite{AAM} illustrated that such a braid presentation is unique up to conjugations and stabilization, that is, the Markov Theorem. Other proofs of classical Markov Theorem refer to \cite{DBE, JSB1, JSB2, SL1, SL2, SL3, HRM, P.T, NW}.

In \cite{LHK}, Kauffman introduced virtual braid group, which is the extension of classical braid group introduced by Artin in \cite{EA}, and it can be described by braid generators and braid relations.
In 2000, Kamada in preprint of \cite{Kamada2007} proved the Alexander Theorem for virtual links and showed the Markov Theorem for virtual braids combining with virtual braid groups, giving a set of global moves on virtual braids that generate the same equivalence classes as the virtual link types of their closures.
In \cite{Kauffman2004}, Kauffman and Lambropoulou gave a new method for converting virtual links to virtual braids and proved every (oriented) virtual link can be represented by a virtual braid whose closure is isotopic to the original link.
This is the analogue of the Alexander Theorems for classical braids and links.
And meanwhile Kauffman and Lambropoulou \cite{Kauffman2004} gave reduced presentations for the virtual braid group and the flat virtual braid group (as well as for other categories). Soon after, Kauffman and Lambropoulou \cite{Kauffman2005} followed L-move methods of Lambropoulou and Rourke in \cite{SL1} to prove the virtual Markov Theorems. One benefit of this approach is a fully local algebraic formulation of the theorems in each category. This is the analogue of the Markov Theorems for classical braids and links.
These theorems are important for understanding the structure and classification of virtual links.
The study of virtual links is closely related to the study of flat virtual links \cite{Kauffman2004, Kauffman2005}.

In $2008$, M.O. Bourgoin \cite{M.O.B} generalized virtual links to twisted links.
%It is well known that the equivalence class set of the classical link is embedded into the set of the virtual link, but the set of equivalence classes of virtual links is not embedded into the set of twisted links, that is, there is that two virtual links are equivalent as twisted links, and they are not necessarily equivalent as virtual links.
Recently, S. Kamada and N. Kamada discussed when two virtual links are equivalent as twisted links, and gave a necessary and sufficient condition for this to be the case in \cite{Kamada2020}.
In this paper, we shall generalize the Alexander Theorem and the Markov Theorem for virtual links to twisted links and give reduced presentations for the twisted braid group and the flat twisted braid group.

The rest of this paper is organized as follows.
In section $2$, we introduce the twisted knot theory.
In section $3$, we prove the Alexander Theorem of twisted links and flat twisted links, respectively.
In section $4$, we show the Markov Theorem for the twisted braids and flat twisted braids, respectively.
In section $5$, we give a reduced presentation for the twisted braid group and the flat twisted braid group, respectively.
In section $6$, we give the Algebraic Markov Theorem for the twisted braids and the flat twisted braids, respectively.

\section{Twisted Knot Theory}

 \begin{figure}[!htbp]
  \centering
    \subfigure[$ classical \ knot$]{
  \includegraphics[width=0.2\textwidth]{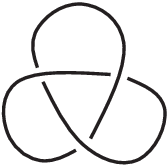}}  \
  \
   \subfigure[$  virtual \ knot$]{
  \includegraphics[width=0.2\textwidth]{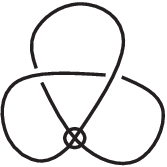}}  \
  \
  \subfigure[$  twisted \ knot$]{
  \includegraphics[width=0.2\textwidth]{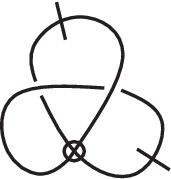}}  \
  \
   \subfigure[$  flat \ twisted \ knot$]{
  \includegraphics[width=0.2\textwidth]{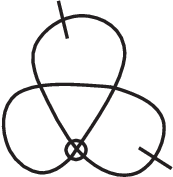}}  \
  \caption{Examples}\label{FL}
\end{figure}

 Virtual knot theory is a generalization of classical knot theory, and twisted knot theory is an extension of virtual knot theory.
 A classical link can be represented by a \emph{link diagram}, which is a 4-regular plane graph with extra structure at its vertices (see Figure \ref{FL}(a)).
A virtual link is a link which is embedded on thickened orientable surface.
A \emph{virtual link diagram} is a link diagram which may have virtual crossings, which are encircled crossings without over-under information (see Figure \ref{FL}(b)).
A \emph{virtual link} is an equivalence class of a virtual link diagram by \emph{Reidemeister moves} $R1$, $R2$, $R3$ and \emph{virtual Reidemeister moves} $V1$, $V2$, $V3$, $V4$ in Figure \ref{R}. All of these are called \emph{generalized Reidemeister moves}.

\begin{figure}[!htbp]
  \centering
  \includegraphics[width=1\textwidth]{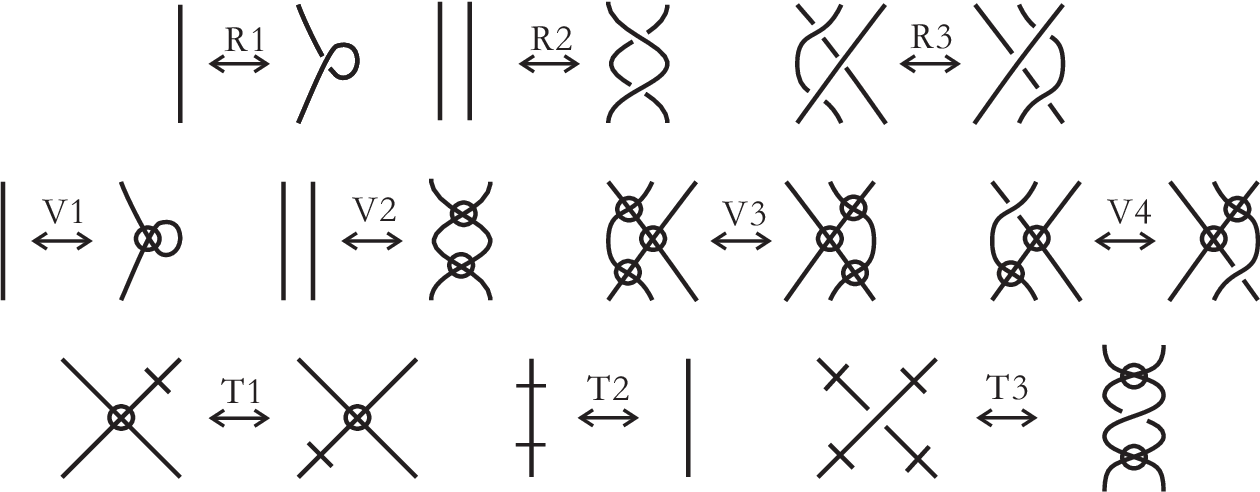}\\
  \caption{Extended Reidemeister moves}\label{R}
\end{figure}

In $2008$, M.O. Bourgoin \cite{M.O.B} generalized virtual links to twisted links. A twisted link is a link which is embedded on thickened orientable surface or thickened non-orientable surface.
A \emph{twisted link diagram} is a virtual link diagram which may have some bars (see Figure \ref{FL}(c)).
A \emph{twisted link} is an equivalence class of a twisted link diagram by \emph{Reidemeister moves} $R1$, $R2$, $R3$, \emph{virtual Reidemeister moves} $V1$, $V2$, $V3$, $V4$ and \emph{twisted Reidemeister moves} $T1$, $T2$, $T3$ in Figure \ref{R}. All of these are called \emph{extended Reidemeister moves}.

\begin{figure}[!htbp]
  \centering
   \subfigure[]{
  \includegraphics[width=0.5\textwidth]{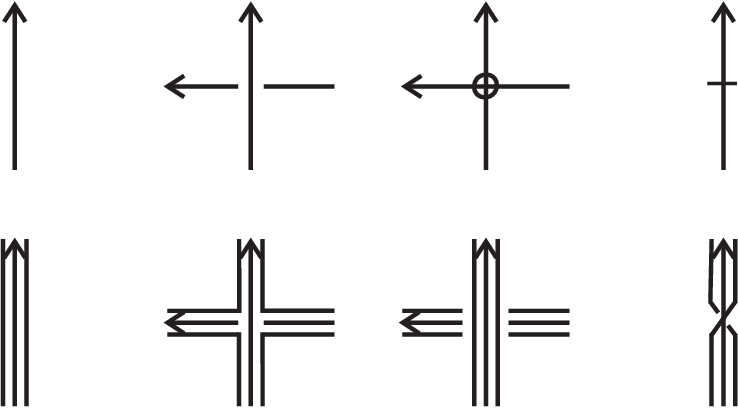}}~~~~~~~~
   \subfigure[]{
  \includegraphics[width=0.26\textwidth]{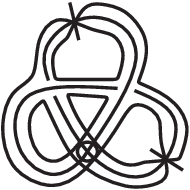}}
  \caption{Abstract link diagram}\label{ABS}
\end{figure}

A geometric interpretation for a twisted link diagram $D$ is obtained by considering its associated abstract link diagram $A(D)$ \cite{Kamada2020}. For a twisted link diagram, a regular neighborhood of the link is an \emph{abstract link diagram}. At a free arc, a regular neighborhood boundary moves forward along the arcs. At a classical crossing, a regular neighborhood boundary turns so as to avoid crossing the link diagram. At a virtual crossing, a regular neighborhood boundary goes through the virtual crossing. At a bar, a regular neighborhood boundary crosses to the other side of the link diagram. Note that a bar of $D$ implies a half-twisted of the ambient surface of $A(D)$ \cite{M.O.B} (see Figure \ref{ABS}). Obviously, the abstract link diagrams associated with extended Reidemeister moves are unchanged and orientation of link does't affect the bars. It should be pointed that there is no analogue of move $T1$ with a real crossing due to abstract link diagram.

Note that we can use the virtual Reidemeister moves or detour move even though there are bars on some arcs based on move $T1$ as shown in Figure \ref{DE}.

\begin{figure}[!htbp]
  \centering
  \includegraphics[width=0.7\textwidth]{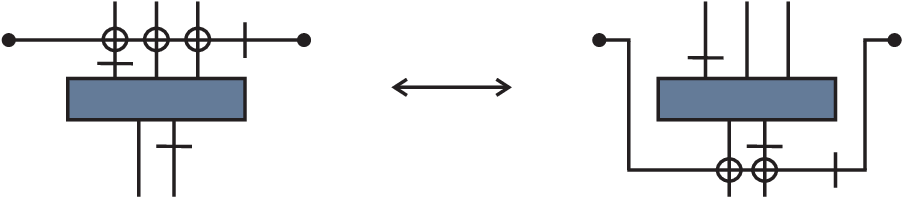}\\
  \caption{Detour move with bars}\label{DE}
\end{figure}

Goussarov, Polyak and Viro \cite{M.M.O} observed that there are two `` forbidden moves " $F1$ and $F2$ (see Figure \ref{FM1}) on virtual knot diagram. Any virtual knot diagram can be deformed into a trivial knot diagram when these two moves are allowed \cite{T.K, S.N}. We also discussed `` forbidden moves " in twisted knot theory in \cite{XD}.

\begin{figure}[!htbp]
  \centering
  \includegraphics[width=0.7\textwidth]{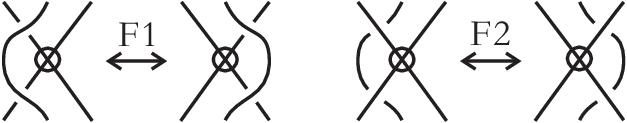}\\
  \caption{``Forbidden moves" $F1$ and $F2$}\label{FM1}
\end{figure}

 A \emph{virtual braid} on $n$ strands is a braid on $n$ strands in the classical sense, which may also contain virtual crossings. A \emph{twisted braid} on $n$ strands is a braid on $n$ strands in the virtual sense, which may have some bars on arcs.
The \emph{closure of a twisted braid} is obtained by joining with simple arcs the corresponding endpoints of the braid on its plane (see Figure \ref{C}).

\begin{figure}[!htbp]
  \centering
  \includegraphics[width=0.45\textwidth]{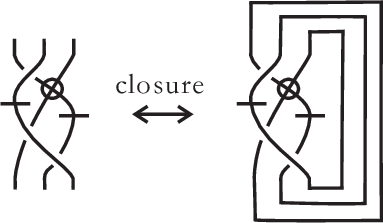}\\
  \caption{Closure}\label{C}
\end{figure}

The set of isotopy classes of virtual braids on $n$ strands forms a group, \emph{the virtual braid group}, denoted $\mathcal{VB}_{n}$, that can be described by generators and relations, generalizing the generators and relations of the classical braid group, introduced by Artin in \cite{EA}. In consideration of the twisted braid, we define \emph{the twisted braid group} denoted $\mathcal{TB}_{n}$ as a group which is the set of isotopy classes of twisted braids on $n$ strands.

\emph{A flat crossing} is a real crossing without over and under information (see Figure \ref{FL}(d)). \emph{A flat virtual link diagram} is a diagram with flat crossings and virtual crossings. \emph{Two flat virtual link diagrams are equivalent} if there is a sequence of
\emph{generalized flat virtual Reidemeister moves} taking one to the other. These moves are shown in
Figure \ref{R} R1-R3 and V1-V4, but with flat crossings taking the place of classical crossings.
Let \emph{a flat twisted link diagram} be a flat virtual link diagram which may have some bars (see Figure \ref{FL}(d)). \emph{Two flat twisted link diagrams are equivalent} if there is a sequence of \emph{extended flat twisted Reidemeister moves} taking one to the other. These moves are shown in Figure \ref{R}, but with flat crossings taking the place of classical crossings. We refer the reader to \cite{Kauffman2004, Kauffman2005} about recent study of flat virtual knots and links.
That is, \emph{a flat twisted link} is an equivalence class of a flat twisted link diagram by extended flat twisted Reidemeister moves.
Similarly, the flat twisted links may play an important role in understanding the twisted links.
And the flat twisted links have itself characteristics and properties.

\section{The Alexander Theorem}

In 2004, Kauffman and Lambropoulou \cite{Kauffman2004} gave a new method for converting virtual links to virtual braids. Indeed, the braiding method given by Kauffman and Lambropoulou \cite{Kauffman2004} is quite general and applies to all the categories in which braiding can be accomplished. Kauffman and Lambropoulou \cite{Kauffman2004} proved following Theorem \ref{virtual} by using the braiding method and some braiding techniques. In \cite{Kauffman2005}, Kauffman and Lambropoulou illusrated the braiding chart for crossings and the basic braiding move for the free up-arc by a different braiding algorithm and also proved following Theorem \ref{virtual}. In this section, we shall extend the virtual Alexander Theorem to twisted braids, in fact, to all the categories in which braids are constructed by using the braiding algorithm.

\begin{thm}(\cite{Kauffman2004}, Theorem 1)\label{virtual}
Every (oriented) virtual link can be represented by a virtual braid whose closure is isotopic to the original link.
\end{thm}

For the twisted links, we have a similar statement as follows.

\begin{thm}\label{twistedbraid}
Every (oriented) twisted link can be represented by a twisted braid whose closure is isotopic to the original link.
\end{thm}

\begin{proof}
Given a (oriented) twisted link diagram, we can take some orientation, and the oriented twisted link diagram is obtained. Then we use the braiding algorithm in \cite{Kauffman2005}, that is, the down-arcs (an arc with a downward direction) will stay in place while the up-arcs (an arc with an upward direction) shall be eliminated.
 Now, an up-arc will either be an arc of a crossing (called valid crossing) or it will be a free up-arc (a simple arc with an upward direction). The braiding algorithm is as follows.

 Firstly, we find all valid crossings and all free up-arcs in the twisted link diagram. Then we use numerical labels to mark the endpoints of valid crossings and free up-arcs, and get some marked crossings and marked free up-arcs.

 Now, we intend to braid these valid crossings and free up-arcs, respectively.
 For valid crossings but without bars, we refer to the braiding chart of Figure 7 in \cite{Kauffman2005} (see Appendix Figure \ref{braiding1}).
 Except for the local crossings shown in the illustrations, all other crossings of the new braid strands with the rest of the diagram are virtual. This is indicated abstractly by placing virtual crossings at the ends of the new strands.
 And for valid crossings with bars on arcs, only nine charts are listed (see Appendix Figure \ref{braiding2}-\ref{braiding10}), others are similar with the following severe cases.

\begin{figure}[!htbp]
  \centering
  \includegraphics[width=0.8\textwidth]{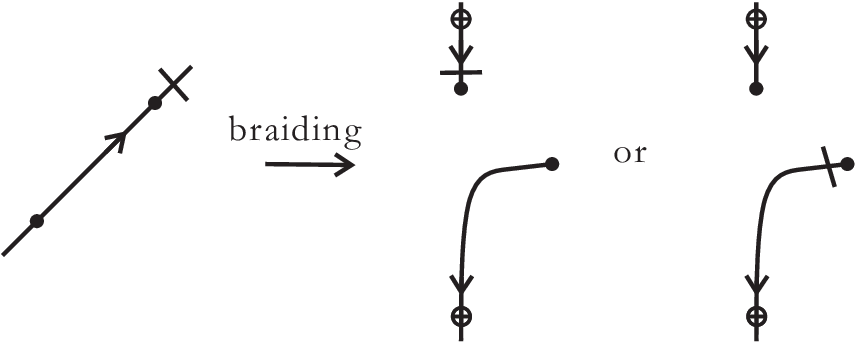}\\
  \caption{A free up-arc with bar}\label{braiding}
\end{figure}

It remains to braid the free up-arcs. The braiding method of the free up-arcs without bars refer to \cite{Kauffman2005} (see Figure \ref{braiding}, but without bars). For a free up-arc with bars, we place these bars in anywhere in the new braid strand, by $T1$ and since all new crossings are virtual (see Figure \ref{braiding}). Finally, we only need to connect the two endpoints corresponding to the same number label, and make sure that the top and bottom ends of the new braid strands correspond one-to-one.

At this time, we create a series of corresponding braid strands and we have one up-arc less. Note that joining the two corresponding braid strands yields a twisted tangle diagram obviously isotopic to the starting one.
\end{proof}

\begin{figure}[!htbp]
  \centering
   \subfigure[]{
  \includegraphics[width=0.25\textwidth]{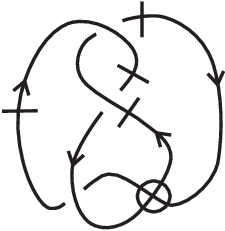}}   ~~~~~~~~~~~~~~~~ \
   \subfigure[]{
  \includegraphics[width=0.25\textwidth]{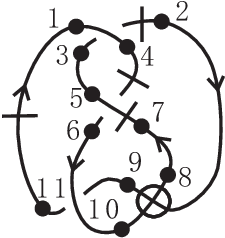}}
  \caption{Twisted link diagram} \label{EX1}
\end{figure}

In Figure \ref{EX1}-\ref{EX3}, we illustrate an example by the above braiding method. Note that the position of bars is not unique due to the detour move with bars on arcs.

We give a oriented twisted link diagram (see Figure \ref{EX1}(a)), which contains 3 classical crossings, 1 virtual crossing and 4 bars. Observe that there are $3$ valid crossings (2 classical crossings and 1 virtual crossings) and 3 free up-arcs. Now, numbers are used to mark endpoints of 3 valid crossings and 3 free up-arcs (see Figure \ref{EX1}(b)). Figure \ref{EX2} showed braid structures corresponding to 3 valid crossings and 3 free up-arcs according to braid chart (see Figure \ref{braiding} and Appendix Figure \ref{braiding1}-\ref{braiding10}). Finally, we connect the same number label to get a twisted braid diagram (see Figure \ref{EX3}). It is easy to verify that the closure of the twisted braid diagram isotopic to the starting one. Note that the crossings not containing up-arcs remains.

\begin{figure}[!htbp]
  \centering
   \subfigure[]{
  \includegraphics[width=0.7\textwidth]{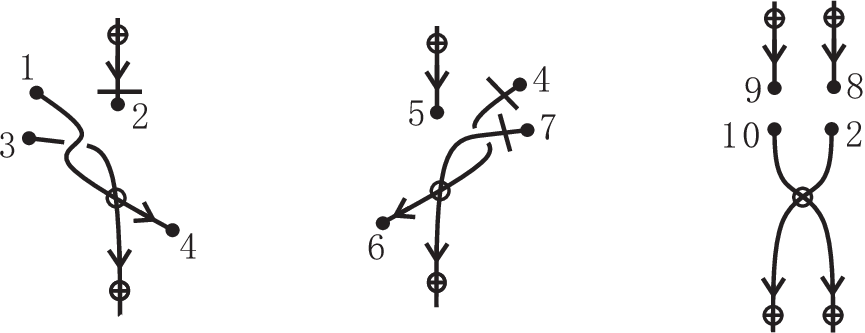}}   \
   \subfigure[]{
  \includegraphics[width=0.7\textwidth]{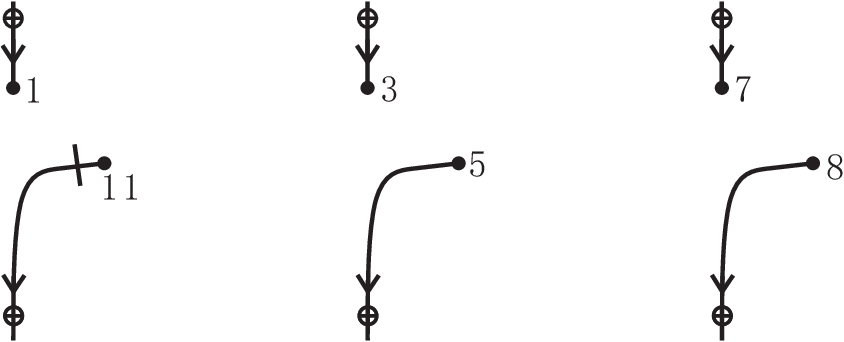}}  \
  \caption{Braiding structures corresponding to 3 valid crossings and 3 free up-arcs}  \label{EX2}
\end{figure}

\begin{figure}[!htbp]
  \centering
  \includegraphics[width=0.3\textwidth]{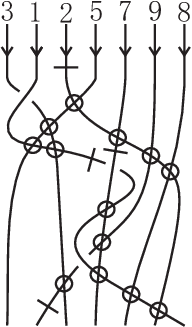}\\
  \caption{Twisted braid diagram} \label{EX3}
\end{figure}

For the flat twisted links, we have.

\begin{thm}
Every (oriented) flat twisted link can be represented by a flat twisted
braid whose closure is isotopic to the original link.
\end{thm}

\begin{proof}
The proof of the above theorem is similar to that of the Theorem \ref{twistedbraid}, but with flat crossings taking the place of classical crossings.
\end{proof}

\section{The Markov Theorem}

It is natural that different choices shall result in different twisted braids when applying the braiding algorithm as well as local isotopy changes on the diagram level. But the Markov Theorem tells us that such a presentation is unique up to certain basic moves for virtual braid case.

\subsection{The Markov Theorem For Twisted Braids}
In 2006, Kauffman and Lambropoulou in \cite{Kauffman2005} proved following Theorem \ref{Markov} by using the L-move methods of Lambropoulou and Rourke in \cite{SL1}. In this section, we generalize virtual Markov Theorem (Theorem \ref{Markov}) to twisted links by applying the L-move methods.

\begin{figure}[!htbp]
  \centering
    \subfigure[Virtual and real conjugation]{
  \includegraphics[width=0.8\textwidth]{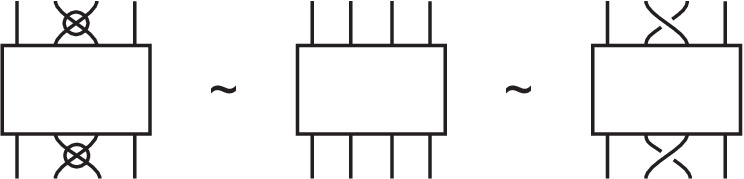}}  \
   \subfigure[Left and right virtual $L_{v}$-moves]{
  \includegraphics[width=0.8\textwidth]{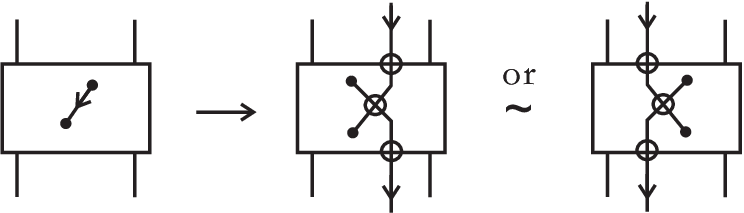}}   \
  \subfigure[Right and left real $L_{v}$-moves]{
  \includegraphics[width=0.8\textwidth]{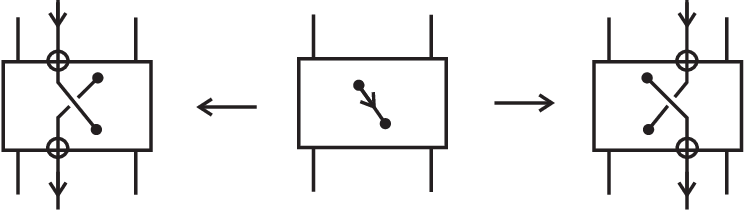}}   \
  \subfigure[Right and left under-threaded $L_{v}$-moves]{
  \includegraphics[width=0.8\textwidth]{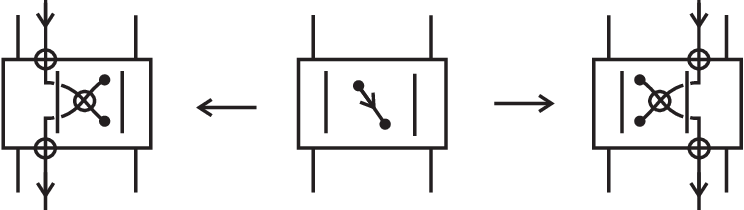}}   \
  \caption{The moves appear in (1)-(4) of Theorem \ref{Markov}}\label{LV}
\end{figure}

\begin{figure}[!htbp]
  \centering
  \includegraphics[width=0.6\textwidth]{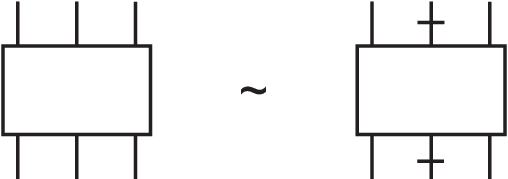}\\
  \caption{Twisted conjugation}\label{TWC}
\end{figure}

\begin{thm}(\cite{Kauffman2005}, Theorem 2)\label{Markov}
Two oriented virtual links are isotopic if and only if any two corresponding virtual braids differ by virtual braid isotopy and a finite sequence of the following moves or their inverses:
\begin{itemize}
  \item [(1)] Real conjugation.
  \item [(2)] Right virtual $L_{v}$-moves.
  \item [(3)] Right real $L_{v}$-moves.
  \item [(4)] Right and left under-threaded $L_{v}$-moves.
\end{itemize}
\end{thm}

 The above moves are shown in Figure \ref{LV}. Definitions of these moves refer to \cite{Kauffman2005}.

In the set of virtual braids, the braid move (1)-(4) and the virtual braid isotopy generate an equivalent
class, called $L$-equivalent.

Note that virtual and real conjugation are closely associated with generalized Reidemeister moves $V2$ and $R2$, respectively. Virtual and real $L_{v}-$move are closely associated with generalized Reidemeister moves $V1$ and $R1$, respectively. Left and right under-threaded $L_{v}-$move are closely associated with generalized Reidemeister moves $V1$ and $R2$.

Note that in the statement of Theorem \ref{Markov}, virtual conjugation and left virtual or real $L_{v}-$moves are not used because Kauffman and Lambropoulou \cite{Kauffman2005} has showed that all these moves follow from virtual braid moves (1)-(4) together with virtual braid isotopy.

For twisted braids, we shall introduce a new move based on twisted Reidemeister moves $T2$, that is, \emph{twisted conjugation} (see Figure \ref{TWC}).
Then we can extend the above statement on virtual links to twisted links.

\begin{thm}\label{UU}
Two oriented twisted links are isotopic if and only if any two corresponding twisted braids differ by twisted braid isotopy and a finite sequence of the following moves or their inverses:
\begin{itemize}
  \item [(1)] Real conjugation and twisted conjugation.
  \item [(2)] Right virtual $L_{v}$-moves.
  \item [(3)] Right real $L_{v}$-moves.
  \item [(4)] Right and left under-threaded $L_{v}$-moves.
\end{itemize}
\end{thm}

 The above moves are shown in Figure \ref{LV} and Figure \ref{TWC}.

\begin{proof}
\begin{figure}[!htbp]
  \centering
   \subfigure[]{
  \includegraphics[width=1\textwidth]{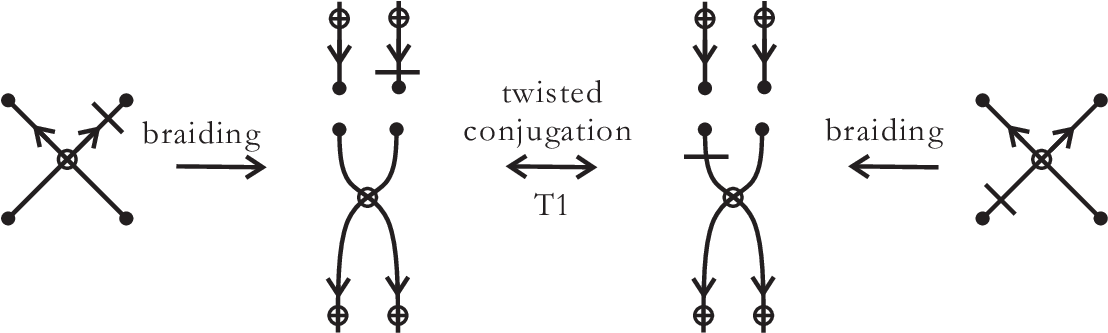}}   \
   \subfigure[]{
  \includegraphics[width=1\textwidth]{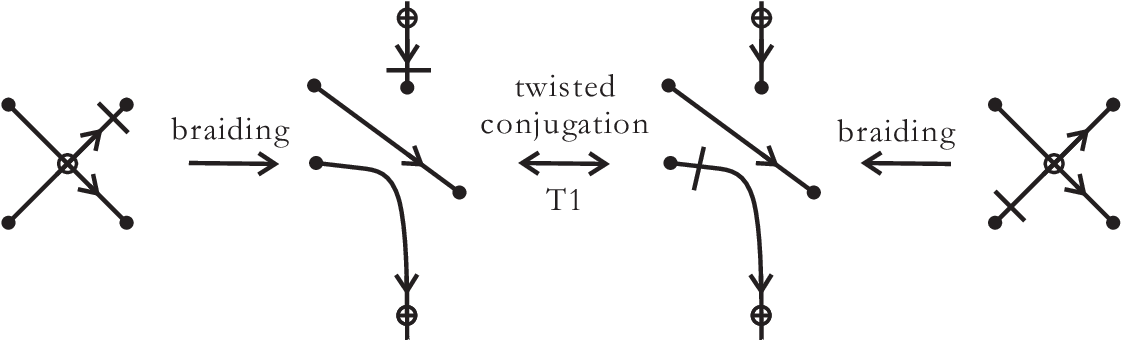}}   \
  \subfigure[]{
  \includegraphics[width=1\textwidth]{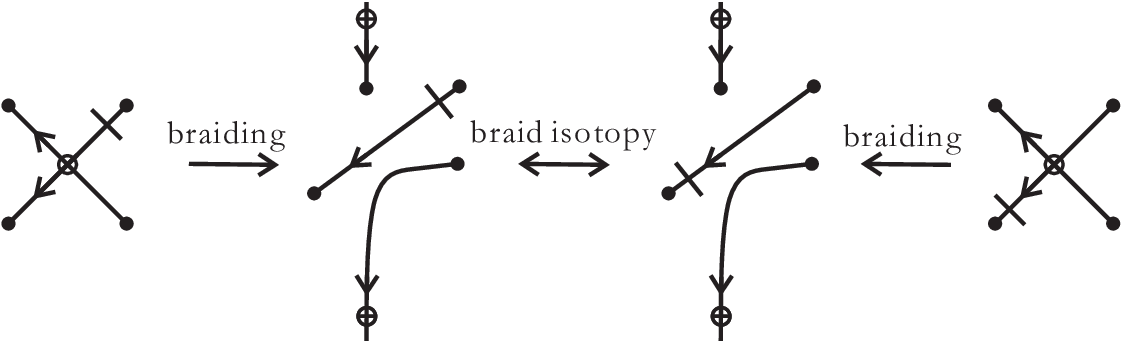}}   \
  \caption{The $T1$ moves with one up-arc less}\label{T1}
\end{figure}

\begin{figure}[!htbp]
  \centering
  \subfigure[]{
  \includegraphics[width=0.7\textwidth]{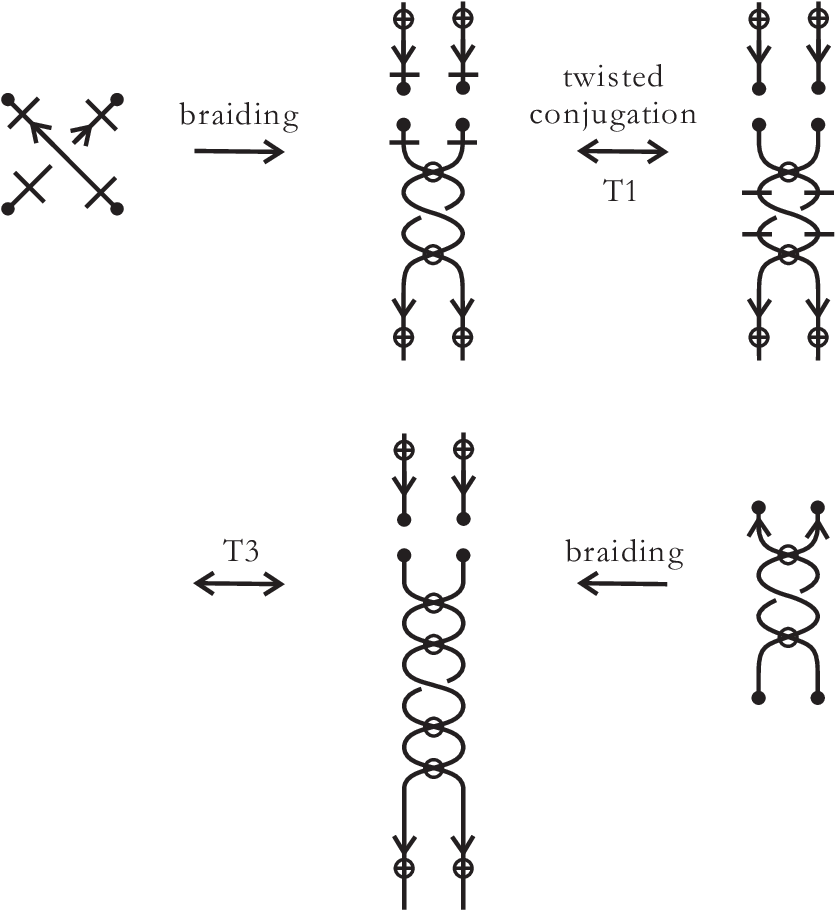}}   \
   \subfigure[]{
  \includegraphics[width=0.7\textwidth]{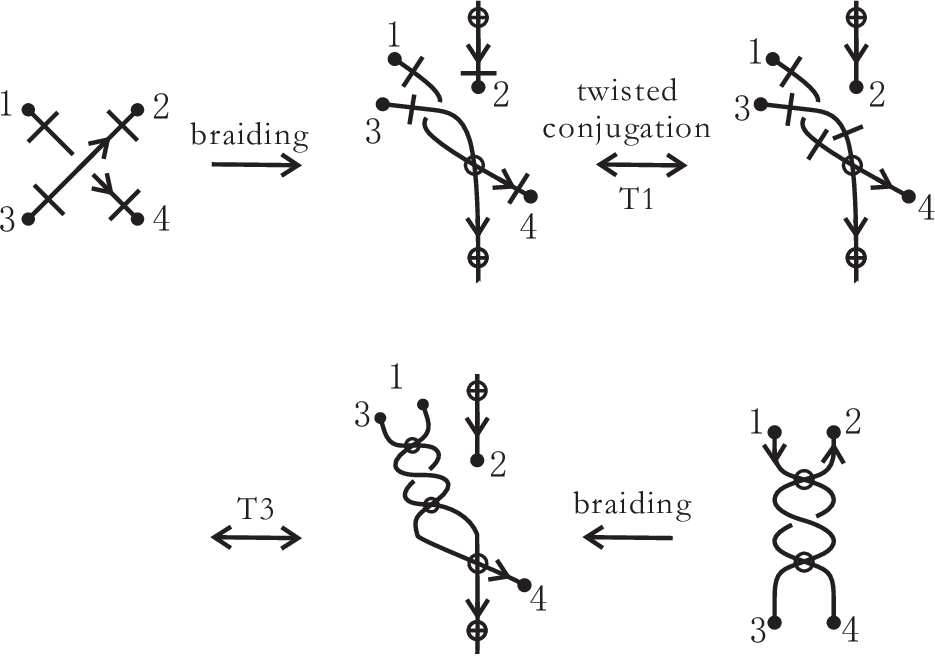}}   \
  \caption{The $T3$ moves with one up-arc less}\label{T3}
\end{figure}

By the definition of twisted conjugation and Theorem \ref{Markov}, sufficiency is obvious. We have to show the converse.

 We follow the proof method of in Theorem \ref{Markov}.
 Firstly, for a given twisted link diagram, one can review the braiding algorithm given above. We number the valid crossings and free up-arcs in the twisted link diagram, and braid them by using the braiding chart (see Figure \ref{braiding} and Appendix Figure \ref{braiding1}-\ref{braiding10}). Note that the choice of the endpoints may be different when choosing a valid crossing and a free up-arc because a free up-arc may have two endpoints or other subdividing points between two endpoints. Kauffman and Lambropoulou in \cite{Kauffman2005} showed two virtual braids are $L$-equivalent if these virtual braids differ from different choices of endpoint on a free up-arc (without bars). Obviously, it is also true for the twisted link diagram.

 For a free up-arc with a bar, one can show that the choice of the endpoint is arbitrary by moving the bar out of endpoints and combining the proofs of Kauffman and Lambropoulou in \cite{Kauffman2005}. And the braiding of the crossings are independent, so regardless of braiding order.

 Secondly, twisted link diagram is not unique. For the same link diagram, different direction on the plane will lead to difference of valid crossings and free up-arcs. Kauffman and Lambropoulou in \cite{Kauffman2005} also explained that different braids caused by this difference is $L$-equivalent for virtual links diagram. Obviously, it is also true for the twisted link diagram.

 Finally, for different but isotopy twisted link diagram, they are equivalent under the extended Reidemeister moves. Therefore, it suffices to show that twisted isotopy moves. Kauffman and Lambropoulou in \cite{Kauffman2005} have stated that the virtual braids corresponding to the classical Reidemeister moves $R1$, $R2$, $R3$, virtual Reidemeister moves $V1$, $V2$, $V3$, $V4$ and the special detour move are the $L$-equivalent. Note that the twisted Reidemeister move $T2$ are clearly established.

So we just have to show that the twisted braids corresponding to the twisted Reidemeister moves $T1$ and $T3$ differ by twisted braid isotopy and a finite sequence of the above moves or their inverses..

Likewise, we only need to check the moves that involve up-arcs, as the others follow by twisted braid isotopy. We first discuss moves $T1$ with at least one up-arc (see Figure \ref{T1}).

In the Figure \ref{T3}, we check two types of the moves $T3$. Here it is twisted conjugation that plays a major role.

\end{proof}
\subsection{The Markov Theorem For Flat Twisted Braids}

The \emph{flat virtual braids} were introduced in \cite{Kauffman2000}. The Markov Theorem for flat virtual braids is proved by Kauffman and Lambropoulou in \cite{Kauffman2005}.

\begin{thm}(\cite{Kauffman2005}, Theorem 4)\label{FUU}
Two oriented flat virtual links are isotopic if and only if any two corresponding flat virtual braids differ by flat virtual braid isotopy and a finite sequence of the following moves or their inverses:
\begin{itemize}
  \item [(1)] Flat conjugation.
  \item [(2)] Right virtual $L_{v}$-moves.
  \item [(3)] Right flat $L_{v}$-moves.
  \item [(4)] Right and left flat threaded $L_{v}$-moves.
\end{itemize}
\end{thm}

 The above moves are shown in Figure \ref{LV}, but with flat crossings taking the place of classical crossings.

Similarly, A \emph{flat twisted braid} on $n$ strands is a braid on $n$ strands in the flat virtual sense, which may have some bars on arcs.
The \emph{closure of a flat twisted braid} is obtained by joining with simple arcs the corresponding endpoints of the braid on its plane.

For flat twisted braids, we still need to apply the twisted conjugation (see Figure \ref{TWC}).
Then we can extend the Theorem \ref{FUU} on flat virtual links to flat twisted links.

\begin{thm}
Two oriented flat twisted links are isotopic if and only if any two corresponding flat twisted braids differ by flat twisted braid isotopy and a finite sequence of the following moves or their inverses:
\begin{itemize}
  \item [(1)] Flat conjugation and twisted conjugation.
  \item [(2)] Right virtual $L_{v}$-moves.
  \item [(3)] Right flat $L_{v}$-moves.
  \item [(4)] Right and left flat threaded $L_{v}$-moves.
\end{itemize}
\end{thm}

 The above moves are shown in Figure \ref{LV} and Figure \ref{TWC}, but with flat crossings taking the place of classical crossings.

\begin{proof}
The proof of the above theorem is similar to that of the Theorem \ref{UU}, but with flat crossings taking the place of classical crossings.

\end{proof}

\section{Reduced Presentation For The Twisted Braid Groups}
\subsection{A reduced presentation for the twisted braid group}

\setlength{\parindent}{2em}
The set of isotopy classes of twisted braids on $n$ strands forms a group, \emph{the twisted braid group}, denoted $\mathcal{TB}_{n}$,
The group operation is the usual braid multiplication (form $bb^{'}$ by attaching the bottom strand ends of $b$ to the top strand ends of $b^{'}$).
$\mathcal{TB}_{n}$ is generated by the braid generators $\sigma_i$, the virtual generators $\upsilon_i$ and the twisted generators $b_i$ (see Figure \ref{T}).

\begin{figure}[!htbp]
  \centering
   \subfigure[$\sigma_{i}$]{
  \includegraphics[width=0.28\textwidth]{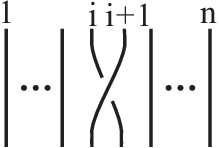}} ~~~
   \subfigure[$v_{i}$]{
  \includegraphics[width=0.28\textwidth]{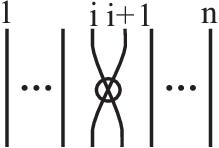}}  ~~~
   \subfigure[$b_{i}$]{
  \includegraphics[width=0.28\textwidth]{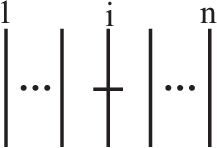}}
  \caption{Generators $\sigma_i$, $\upsilon_i$ and $b_i$ of twisted braid $\mathcal{TB}_{n}$}\label{T}
\end{figure}

Among themselves the generators satisfy the following relations:

braiding relations:
\begin{equation*}
\begin{split}
&\sigma_{i}\sigma_{i}^{-1}=1, \ \ 1 \leq i \leq n-1,\\
&\sigma_{i}\sigma_{i+1}\sigma_{i}=\sigma_{i+1}\sigma_{i}\sigma_{i+1}, \ \ 1 \leq i \leq n-2,\\
&\sigma_{i}\sigma_{j}=\sigma_{j}\sigma_{i}, \  \ |i-j|\geq 2, \\
 \end{split}
\end{equation*}

virtual relations:
\begin{equation*}
\begin{split}
&v_i^2=1,  \ \ 1 \leq i \leq n-1, \\
&v_iv_j=v_jv_i, \  \ |i-j|\geq 2, \\
&v_iv_{i+1}v_i=v_{i+1}v_iv_{i+1}, \  \ 1 \leq i \leq n-2, \\
\end{split}
\end{equation*}

twisted relations:
\begin{equation*}
\begin{split}
                  &b_i^{2}=1, \  \ 1 \leq i \leq n, \\
                  &b_ib_{j}=b_{j}b_i, \  \ i  \neq j, \\
 \end{split}
\end{equation*}

mixed relations:
\begin{equation*}
\begin{split}
 &\sigma_{i}v_j=v_j\sigma_{i}, \  \ |i-j|\geq 2,\\
 &v_i\sigma_{i+1}v_i=v_{i+1}\sigma_iv_{i+1}, \ \ 1 \leq i \leq n-2, \\
 &b_iv_{i}=v_{i}b_{i+1}, \ \ 1 \leq i \leq n-1, \\
 &b_ib_{i+1}\sigma_{i}b_{i+1}b_i=v_{i}\sigma_{i}v_i, \ \ 1 \leq i \leq n-1, \\
 &b_{i}v_{j}=v_{j}b_{i}, \  \ j>i \ or \ j<i-1, \\
 &b_{i}\sigma_{j}=\sigma_{j}b_{i}, \ \ j>i \ or \  j<i-1. \\
 \end{split}
\end{equation*}

Note that the relations
$$b_ib_{i+1}\sigma_{i}b_{i+1}b_i=v_{i}\sigma_{i}v_i, \ \ 1 \leq i \leq n-1 $$

 is equivalent to the relations
 $$b_{i+1}b_i\sigma_{i}b_i b_{i+1}=v_{i}\sigma_{i}v_i, \ \ 1 \leq i \leq n-1. $$

And note that the relations
$$b_iv_{i}=v_{i}b_{i+1}, \ \ 1 \leq i \leq n-1 $$

 is equivalent to the relations
 $$v_{i}b_i=b_{i+1}v_{i}, \ \ 1 \leq i \leq n-1.$$

The \emph{classical braid group} $\mathcal{B}_{n}$ is generated by the braid generators $\sigma_i$.
By an argument in \cite{Kamada2007}, we see that the subgroup of $\mathcal{TB}_{n}$
generated by $\sigma_i (i=1,\cdots, n)$ and unit element is isomorphic to the the classical braid group $\mathcal{B}_{n}$, the subgroup
generated by $v_i (i=1,\cdots, n-1)$ and unit element is isomorphic to the symmetric group $\mathcal{S}_{n}$ and the subgroup
generated by $b_i (i=1,\cdots, n)$ is isomorphic to the $2^{n}$ Abelian group $Z_{2}\bigotimes Z_{2}\bigotimes  \cdots \bigotimes Z_{2}$.

The \emph{virtual braid group} $\mathcal{VB}_{n}$ is generated by the braid generators $\sigma_i$ and the virtual generators $\upsilon_i$ as shown in Figure \ref{T}. In 2004, Kauffman and Lambropoulou \cite{Kauffman2004} gave following reduced presentation for $\mathcal{VB}_{n}$.

\begin{thm}(\cite{Kauffman2004}, Theorem 2)\label{reducedpre}
The virtual braid group $\mathcal{VB}_{n}$ has the following reduced
presentation:
\begin{equation*}\label{1}
\begin{split}
  \mathcal{VB}_{n}=\langle \ \sigma_{1}, v_{1}, \cdots, v_{n-1} \ | \ &v_{i}v_{i+1}v_{i}=v_{i+1}v_{i}v_{i+1}, \ v_{i}v_{j}=v_{j}v_{i}, \ |i-j|\geq 2 ,\\
  &v_{i}^{2}=1, \ 1 \leq i \leq n-1, \ \sigma_{1}v_{j}=v_{j}\sigma_{1}, \ j>2, \\
  &(v_{1}\sigma_{1}v_{1})(v_{2}\sigma_{1}v_{2})(v_{1}\sigma_{1}v_{1})=(v_{2}\sigma_{1}v_{2})(v_{1}\sigma_{1}v_{1})(v_{2}\sigma_{1}v_{2}), \\
 &\sigma_{1}(v_{2}v_{3}v_{1}v_{2}\sigma_{1}v_{2}v_{1}v_{3}v_{2})=(v_{2}v_{3}v_{1}v_{2}\sigma_{1}v_{2}v_{1}v_{3}v_{2})\sigma_{1}  \ \rangle.
  \end{split}
\end{equation*}
\end{thm}

Below we give a reduced presentation for $\mathcal{TB}_{n}$ with generators
\begin{center}
$\{\sigma_{1}, b_{1}, v_{1}, \cdots, v_{n-1}\}.$
\end{center}

In\cite{Kauffman2004}, Kauffman and Lambropoulou introduced following the defining relations:
$$ \sigma_{i+1}:=(v_{i} \cdots v_{2}v_{1})(v_{i+1}\cdots v_{3}v_{2})\sigma_{1}(v_{2}v_{3}\cdots v_{i+1})(v_{1} v_{2}\cdots v_{i}), \ \ i=1, \cdots, n-2. \eqno{(\dag)}$$

According to mixed relations $b_iv_{i}=v_{i}b_{i+1}$, we assume the defining relations:
$$ b_{i+1}:=(v_{i}v_{i-1}\cdots v_1)b_1(v_1v_{2}\cdots v_{i}),\ \ i=1, \cdots, n-1. \eqno{(*)}$$

In the proof of Theorem \ref{twisted braid} below we use the following virtual braid relations, which are simple inferences of the virtual relations $v_iv_j=v_jv_i,~|i-j|\geq 2$ and $v_iv_{i+1}v_i=v_{i+1}v_iv_{i+1},~1 \leq i \leq n-2$:
$$ v_{i}v_{i-1} \cdots v_{2}v_{1}v_{2} \cdots v_{i-1}v_{i}=v_{1}v_{2} \cdots v_{i-1}v_{i}v_{i-1} \cdots v_{2}v_{1}, \ \ i=1, \cdots, n. \eqno{(\ddag)}$$

To give a reduced presentation for $\mathcal{TB}_{n}$, we consider the relations involving $b_{i}$.

\begin{itemize}
   \item For the mixed relations $b_ib_{i+1}\sigma_{i}b_{i+1}b_i=v_{i}\sigma_{i}v_i, \ 1 \leq i \leq n-1$ (or $b_{i+1}b_i\sigma_{i}b_i b_{i+1}=v_{i}\sigma_{i}v_i, \ 1 \leq i \leq n-1$), we only need to keep $b_2b_{1}\sigma_{1}b_{1}b_2=v_{1}\sigma_{1}v_1$, i.e. $(v_{1}b_{1}v_{1})b_1\sigma_{1}b_1(v_{1}b_{1}v_{1})=v_{1}\sigma_{1}v_1$.
   \item For the twisted relations $b_i^{2}=1, \ 1 \leq i \leq n $, we only need to keep $b_{1}^{2}=1$.

   \item For the twisted relations $b_ib_{j}=b_{j}b_i, \ i  \neq j$, we only need to keep $b_{1}b_{2}=b_{2}b_{1}$, i.e. $b_{1}(v_{1}b_{1}v_{1})=(v_{1}b_{1}v_{1})b_{1}$.

   \item For the mixed relations $b_{i}v_{j}=v_{j}b_{i}, \ j>i \ or \ j<i-1$, we only need to keep $b_{1}v_{j}=v_{j}b_{1}, \ j>1$.

   \item For the mixed relations $b_{i}\sigma_{j}=\sigma_{j}b_{i}, \ j>i \ or \  j<i-1 $, we only need to keep $b_{1}(v_{1}v_{2}\sigma_{1}v_{2}v_{1})=(v_{1}v_{2}\sigma_{1}v_{2}v_{1})b_{1}$, respectively.

\end{itemize}

Note that the relation
$$ b_{1}(v_{1}v_{2}\sigma_{1}v_{2}v_{1})=(v_{1}v_{2}\sigma_{1}v_{2}v_{1})b_{1}$$

 is equivalent to the relation
 $$\sigma_{1}(v_{2}v_{1}b_{1}v_{1}v_{2})=(v_{2}v_{1}b_{1}v_{1}v_{2})\sigma_{1}.$$

 Then, we have:

\begin{thm} \label{twisted braid}
The twisted braid group $\mathcal{TB}_{n}$ has the following reduced
presentation:
\begin{equation*}\label{3}
\begin{split}
  \mathcal{TB}_{n}=\langle \ \sigma_{1}, b_{1}, v_{1}, \cdots, v_{n-1} \ | \  &v_{i}v_{i+1}v_{i}=v_{i+1}v_{i}v_{i+1}, \ v_{i}v_{j}=v_{j}v_{i}, \  |i-j|\geq 2 ,\\
  &v_{i}^{2}=1, \ 1 \leq i \leq n-1, \ \sigma_{1}v_{j}=v_{j}\sigma_{1}, \  j>2 ,\\
  &b_{1}^{2}=1, \ b_{1}v_{j}=v_{j}b_{1}, \ j>1\  \\
              &(v_{1}\sigma_{1}v_{1})(v_{2}\sigma_{1}v_{2})(v_{1}\sigma_{1}v_{1})=(v_{2}\sigma_{1}v_{2})(v_{1}\sigma_{1}v_{1})(v_{2}\sigma_{1}v_{2}),\\
                      &\sigma_{1}(v_{2}v_{3}v_{1}v_{2}\sigma_{1}v_{2}v_{1}v_{3}v_{2})=(v_{2}v_{3}v_{1}v_{2}\sigma_{1}v_{2}v_{1}v_{3}v_{2})\sigma_{1},\\
             & b_{1}(v_{1}b_{1}v_{1})=(v_{1}b_{1}v_{1})b_{1},\\
              &\sigma_{1}(v_{2}v_{1}b_{1}v_{1}v_{2})=(v_{2}v_{1}b_{1}v_{1}v_{2})\sigma_{1}, \\
              &(v_{1}b_{1}v_{1})b_1\sigma_{1}b_1(v_{1}b_{1}v_{1})=v_{1}\sigma_{1}v_1 \ \rangle.\\
\end{split}
 \end{equation*}

\end{thm}

\begin{proof}
{\rm \textcircled{1}} \ The braiding relations and mixed relations:
\begin{equation*}\label{4}
\begin{split}
  &\sigma_{i}v_{j}=v_{j}\sigma_{i}, \ \ |i-j|\geq 2, \\
  &\sigma_{i}\sigma_{i+1}\sigma_{i}=\sigma_{i+1}\sigma_{i}\sigma_{i+1}, \  \ 1 \leq i \leq n-2,\\
  &\sigma_{i}\sigma_{j}=\sigma_{j}\sigma_{i}, \  \ |i-j|\geq 2,  \\
  &v_i\sigma_{i+1}v_i=v_{i+1}\sigma_iv_{i+1}, \  \ 1 \leq i \leq n-2, \\
  \end{split}
\end{equation*}

follow from the Theorem \ref{reducedpre}.

{\rm \textcircled{2}} \ The mixed relations $v_{i}b_{i+1}=b_{i}v_{i}$ for $i \geq 1$ follow from the defining relations $(*)$ and virtual relations $v_{i}v_{i}=1_{n}, \ 1 \leq i \leq n-1$.
\begin{equation*}\label{5}
\begin{split}
  v_{i}b_{i+1} &\overset{(*)}=\underline{v_{i}}(v_{i} \cdots v_{1})b_{1}(v_{1} \cdots v_{i}) \\
               &=\underline{(v_{i-1} \cdots v_{1})b_{1}(v_{1} \cdots v_{i-1})}v_{i} \\
               &\overset{(*)}=b_{i}v_{i}. \\
  \end{split}
\end{equation*}

{\rm \textcircled{3}} \ The mixed relations $b_{i}v_{j}=v_{j}b_{i}$ for $ j>i$ or $j<i-1$ follow from the defining relations $(*)$, virtual relations $v_{i}v_{i+1}v_{i}=v_{i+1}v_{i}v_{i+1}, \ v_{i}v_{j}=v_{j}v_{i}, \  |i-j|\geq 2$ and mixed relations $ b_{1}v_{j}=v_{j}b_{1}, \ j>1$.

If $j>i$, we have:
\begin{equation*}\label{6}
\begin{split}
  b_{i}v_{j} &\overset{(*)}=(v_{i-1}\cdots v_{1})b_{1}(v_{1}\cdots v_{i-1})\underline{v_{j}} \\
             &=v_{j}\underline{(v_{i-1}\cdots v_{1})b_{1}(v_{1}\cdots v_{i-1})} \\
             &\overset{(*)}=v_{j}b_{i}.
  \end{split}
\end{equation*}

If $j < i-1$, we have:
\begin{equation*}\label{7}
\begin{split}
  b_{i}v_{j} &\overset{(*)}=(v_{i-1}\cdots v_{1})b_{1}(v_{1}\cdots v_{i-1})\underline{v_{j}} \\
             &=(v_{i-1}\cdots v_{1})b_{1}(v_{1} \cdots v_{j-1}\underline{v_{j}v_{j+1}v_{j}}v_{j+2} \cdots v_{i-1}) \\
             &=(v_{i-1}\cdots v_{1})b_{1}(v_{1} \cdots v_{j-1}\underline{v_{j+1}}v_{j}v_{j+1}v_{j+2} \cdots v_{i-1}) \\
             &=(v_{i-1}\cdots v_{j+2}\underline{v_{j+1}v_{j}v_{j+1}}v_{j-1}\cdots v_{1})b_{1}(v_{1} \cdots v_{j-1}v_{j}v_{j+1}v_{j+2} \cdots v_{i-1}) \\
             &=(v_{i-1}\cdots v_{j+2}\underline{v_{j}}v_{j+1}v_{j}v_{j-1}\cdots v_{1})b_{1}(v_{1} \cdots v_{j-1}v_{j}v_{j+1}v_{j+2} \cdots v_{i-1}) \\
             &=v_{j}\underline{(v_{i-1}\cdots v_{j+2}v_{j+1}v_{j}v_{j-1}\cdots v_{1})b_{1}(v_{1} \cdots v_{j-1}v_{j}v_{j+1}v_{j+2} \cdots v_{i-1})} \\
             &\overset{(*)}=v_{j}b_{i}.
  \end{split}
\end{equation*}

{\rm \textcircled{4}} \ The mixed relations $b_{i}b_{i+1}\sigma_{i}b_{i+1}b_{i}=v_{i}\sigma_{i}v_i$ for $i>1$ follow from the above defining relations $(*)$, $(\dag)$ and $(\ddag)$, the reduced relation $(v_{1}b_{1}v_{1})b_1\sigma_{1}b_1(v_{1}b_{1}v_{1})=v_{1}\sigma_{1}v_1$, the virtual relations $v_{i}v_{i}=1_{n}$ and mixed relations $b_{i}v_{i}=v_{i}b_{i+1}$ for $i \geq 1$ in $\textcircled{2}$ and $b_{i}v_{j}=v_{j}b_{i}$ for $ j>i$ or $j<i-1$ in $\textcircled{3}$.

\begin{equation*}\label{2}
\begin{split}
  b_ib_{i+1}\sigma_{i}b_{i+1}b_i&\overset{(*)(\dag)}=\underline{b_{i}(v_{i}} \cdots v_{1})b_{1}\underline{(v_{1} \cdots v_{i})(v_{i-1} \cdots v_{1})}(v_{i} \cdots v_{3}v_{2})\sigma_{1}(v_{2}v_{3} \cdots v_{i}) \\& \underline{(v_{1} \cdots v_{i-1})(v_{i}\cdots v_{1})}b_{1}(v_{1} \cdots \underline{v_{i})b_{i}} \\
           &\overset{(\ddag)\textcircled{2}}=(v_{i}\underline{b_{i+1}} \cdots v_{1})b_{1}(v_{i} \cdots v_{2}v_{1}\underline{v_{2} \cdots v_{i})(v_{i} \cdots v_{3}v_{2})}\sigma_{1} \\& \underline{(v_{2}v_{3} \cdots v_{i})(v_{i} \cdots v_{2}}v_{1}v_{2} \cdots v_{i})b_{1}(v_{1} \cdots \underline{b_{i+1}}v_{i}) \\
           &\overset{\textcircled{3}}=(v_{i}\cdots v_{1})b_{1}\underline{b_{i+1} (v_{i}} \cdots v_{2}v_{1})\sigma_{1}(v_{1}v_{2} \cdots \underline{v_{i})b_{i+1}}b_{1}(v_{1} \cdots v_{i}) \\
           &\overset{\textcircled{2}}=(v_{i}\cdots v_{1})b_{1} (v_{i}\underline{b_{i}v_{i-1}} \cdots v_{2}v_{1})\sigma_{1}(v_{1}v_{2} \cdots \underline{v_{i-1}b_{i}}v_{i})b_{1}(v_{1} \cdots v_{i}) \\
           &\overset{\textcircled{2}}= \cdots \\
           &\overset{\textcircled{2}}=(v_{i}\cdots v_{1})\underline{b_{1}}(v_{i}v_{i-1} \cdots v_{2}v_{1})b_{1}\sigma_{1}b_{1}(v_{1}v_{2} \cdots v_{i-1}v_{i})\underline{b_{1}}(v_{1} \cdots v_{i}) \\
           &\overset{\textcircled{3}}=(v_{i}\cdots v_{1})(v_{i}v_{i-1} \cdots v_{2}b_{1}\underline{v_{1})b_{1}}\sigma_{1}\underline{b_{1}(v_{1}}b_{1}v_{2} \cdots v_{i-1}v_{i})(v_{1} \cdots v_{i}) \\
            &\overset{\textcircled{2}}=(v_{i}\cdots v_{1})(v_{i}v_{i-1} \cdots v_{2}b_{1}b_{2}\underline{v_{1})\sigma_{1}(v_{1}}b_{2}b_{1}v_{2} \cdots v_{i-1}v_{i})(v_{1} \cdots v_{i}) \\
             &=(v_{i}\cdots v_{1})(v_{i}v_{i-1} \cdots v_{2}b_{1}b_{2}\underline{(v_{1}b_{1}v_{1})}b_1\sigma_{1}b_1\underline{(v_{1}b_{1}v_{1})}b_{2}b_{1}v_{2} \cdots v_{i-1}v_{i})(v_{1} \cdots v_{i}) \\
            &\overset{(*)}=(v_{i}\cdots v_{1})(v_{i}v_{i-1} \cdots v_{2}\underline{b_{1}b_{2}b_{2}b_1}\sigma_{1}\underline{b_1 b_{2}b_{2}b_{1}}v_{2} \cdots v_{i-1}v_{i})(v_{1} \cdots v_{i}) \\
             &=(v_{i}\underline{v_{i-1}\cdots v_{1})(v_{i}v_{i-1} \cdots v_{2})\sigma_{1}(v_{2}\cdots v_{i-1}v_{i})(v_{1} \cdots v_{i-1}}v_{i}) \\
             &\overset{(\dag)}=v_{i}\sigma_{i}v_{i}. \\
  \end{split}
\end{equation*}

{\rm \textcircled{5}} \ The mixed relations $b_{i}\sigma_{j}=\sigma_{j}b_{i}$ for $ j>i$ or $j<i-1$ follow from the defining relations $(*)$ and $(\dag)$, mixed relations $\sigma_{i}v_{j}=v_{j}\sigma_{i}$ for $ |i-j|\geq 2 $, the reduced relation $\sigma_{1}(v_{2}v_{1}b_{1}v_{1}v_{2})=(v_{2}v_{1}b_{1}v_{1}v_{2})\sigma_{1}$, mixed relations $v_{i}b_{i+1}=b_{i}v_{i}$ for $i \geq 1$ in $\textcircled{2}$ and $b_{i}v_{j}=v_{j}b_{i}$ for $j>i$ or $j<i-1$ in $\textcircled{3}$.

If $j > i $, first we have:
\begin{equation*}\label{8}
\begin{split}
  b_{1}\sigma_{j} &\overset{(\dag)}=\underline{b_{1}}(v_{j-1} \cdots v_{2}v_{1})(v_{j}\cdots v_{3}v_{2})\sigma_{1}(v_{2}v_{3}\cdots v_{j})(v_{1}v_{2} \cdots v_{j-1}) \\
                  &\overset{\textcircled{3}}=(v_{j-1} \cdots v_{2}\underline{b_{1}v_{1}})(v_{j}\cdots v_{3}v_{2})\sigma_{1}(v_{2}v_{3}\cdots v_{j})(v_{1}v_{2} \cdots v_{j-1}) \\
                  &\overset{\textcircled{2}}=(v_{j-1} \cdots v_{2}v_{1}\underline{b_{2}})(v_{j}\cdots v_{3}v_{2})\sigma_{1}(v_{2}v_{3}\cdots v_{j})(v_{1}v_{2} \cdots v_{j-1}) \\
                   &\overset{\textcircled{3}}=(v_{j-1} \cdots v_{2}v_{1})(v_{j}\cdots v_{3}\underline{b_{2}v_{2}})\sigma_{1}(v_{2}v_{3}\cdots v_{j})(v_{1}v_{2} \cdots v_{j-1})\\
                   &\overset{\textcircled{2}}=(v_{j-1} \cdots v_{2}v_{1})(v_{j}\cdots v_{3}
                   v_{2})\underline{b_{3}}\sigma_{1}(v_{2}v_{3}\cdots v_{j})(v_{1}v_{2} \cdots v_{j-1}) \\
                   &\overset{(*)}=(v_{j-1} \cdots v_{2}v_{1})(v_{j}\cdots v_{3}
                   v_{2})\underline{(v_{2}v_{1}b_{1}v_{1}v_{2})\sigma_{1}}(v_{2}v_{3}\cdots v_{j})(v_{1}v_{2} \cdots v_{j-1}) \\
                   &=(v_{j-1} \cdots v_{2}v_{1})(v_{j}\cdots v_{3}
                   v_{2})\sigma_{1}\underline{(v_{2}v_{1}b_{1}v_{1}v_{2})}(v_{2}v_{3}\cdots v_{j})(v_{1}v_{2} \cdots v_{j-1}) \\
                    &\overset{(*)}=(v_{j-1} \cdots v_{2}v_{1})(v_{j}\cdots v_{3}
                   v_{2})\sigma_{1}\underline{b_{3}(v_{2}}v_{3}\cdots v_{j})(v_{1}v_{2} \cdots v_{j-1}) \\
                   &\overset{\textcircled{2}}=(v_{j-1} \cdots v_{2}v_{1})(v_{j}\cdots v_{3}
                   v_{2})\sigma_{1}(v_{2}\underline{b_{2}}v_{3}\cdots v_{j})(v_{1}v_{2} \cdots v_{j-1}) \\
                    &\overset{\textcircled{3}}=(v_{j-1} \cdots v_{2}v_{1})(v_{j}\cdots v_{3}
                   v_{2})\sigma_{1}(v_{2}v_{3}\cdots v_{j})(\underline{b_{2}v_{1}}v_{2} \cdots v_{j-1}) \\
  \end{split}
\end{equation*}
\begin{equation*}
\begin{split}
                   &\overset{\textcircled{2}}=(v_{j-1} \cdots v_{2}v_{1})(v_{j}\cdots v_{3}
                   v_{2})\sigma_{1}(v_{2}v_{3}\cdots v_{j})(v_{1}\underline{b_{1}v_{2}} \cdots v_{j-1}) \\
                   &\overset{\textcircled{3}}=\underline{(v_{j-1} \cdots v_{2}v_{1})(v_{j}\cdots v_{3}
                   v_{2})\sigma_{1}(v_{2}v_{3}\cdots v_{j})(v_{1}v_{2}v_{3} \cdots v_{j-1})}b_{1} \\
                   &\overset{(\dag)}=\sigma_{j}b_{1}. \\
  \end{split}
\end{equation*}

 Then
\begin{equation*} \label{9}
\begin{split}
  b_{i}\sigma_{j} &\overset{(*)}=(v_{i-1} \cdots v_{1})b_{1}(v_{1} \cdots v_{i-1})\underline{\sigma_{j}} \\
                  &=\sigma_{j}\underline{(v_{i-1} \cdots v_{1})b_{1}(v_{1} \cdots v_{i-1})} \\
                  &\overset{(*)}=\sigma_{j}b_{i}. \\
  \end{split}
\end{equation*}

If $j < i-1$, first we have:
\begin{equation*} \label{10}
\begin{split}
\sigma_{1} b_{i} &\overset{(*)}=\underline{\sigma_{1}}(v_{i-1} \cdots v_{1})b_{1}(v_{1} \cdots v_{i-1}) \\
                 &=(v_{i-1} \cdots \underline{\sigma_{1}v_{2}v_{1})b_{1}(v_{1}v_{2}}\cdots v_{i-1}) \\
                 &=(v_{i-1} \cdots v_{2}v_{1})b_{1}(v_{1}v_{2}\underline{\sigma_{1}}\cdots v_{i-1}) \\
                 &=\underline{(v_{i-1} \cdots v_{2}v_{1})b_{1}(v_{1}v_{2}\cdots v_{i-1})}\sigma_{1} \\
                 &\overset{(*)}=b_{i}\sigma_{1}. \\
  \end{split}
\end{equation*}

Then
\begin{equation*} \label{11}
\begin{split}
  b_{i}\sigma_{j} &\overset{(\dag)}=\underline{b_{i}}(v_{j-1} \cdots v_{2}v_{1})(v_{j}\cdots v_{3}v_{2})\sigma_{1}(v_{2}v_{3}\cdots v_{j})(v_{1}v_{2} \cdots v_{j-1}) \\
                  &\overset{\textcircled{3}}=\underline{(v_{j-1} \cdots v_{2}v_{1})(v_{j}\cdots v_{3}v_{2})\sigma_{1}(v_{2}v_{3}\cdots v_{j})(v_{1}v_{2} \cdots v_{j-1})}b_{i} \\
                  &\overset{(\dag)}=\sigma_{j}b_{i}. \\
  \end{split}
\end{equation*}

{\rm \textcircled{6}} \ The twisted relations $b_{i}b_{j}=b_{j}b_{i}$ for $i \neq j$ follow from the defining relations $(*)$, the reduced relation $ b_{1}(v_{1}b_{1}v_{1})=(v_{1}b_{1}v_{1})b_{1}$, mixed relations $b_{i}v_{j}=v_{j}b_{i}$ for $j>i$ or $j<i-1$ in $\textcircled{3}$.

Without loss of generality, we assume $j>i$. First we have:
\begin{equation*} \label{12}
\begin{split}
  b_{1}b_{j} &\overset{(*)}=\underline{b_{1}}(v_{j-1} \cdots v_{1}b_{1}v_{1} \cdots v_{j-1}) \\
             &\overset{\textcircled{3}}=(v_{j-1} \cdots v_{2})\underline{b_{1}v_{1}b_{1}(v_{1}} \cdots v_{j-1}) \\
             &=(v_{j-1} \cdots v_{2})v_{1}b_{1}v_{1}\underline{b_{1}}(v_{2} \cdots v_{j-1}) \\
               &\overset{\textcircled{3}}=\underline{(v_{j-1} \cdots v_{2})v_{1}b_{1}v_{1}(v_{2} \cdots v_{j-1})}b_{1} \\
               &\overset{(*)}=b_{j}b_{1}. \\
  \end{split}
\end{equation*}

Then
\begin{equation*} \label{13}
\begin{split}
  b_{i}b_{j} &\overset{(*)}=(v_{i-1} \cdots v_{1}b_{1}v_{1} \cdots v_{i-1})\underline{b_{j}} \\
             &\overset{\textcircled{3}}=b_{j}\underline{(v_{i-1} \cdots v_{1}b_{1}v_{1} \cdots v_{i-1})} \\
             &\overset{(*)}=b_{j}b_{i}. \\
  \end{split}
\end{equation*}

{\rm \textcircled{7}} \ The twisted relations $b_{i}^{2}=1$ for $i>1$ follow from the defining relations $(*)$, the virtual relations $v_{i}v_{i}=1_{n}$ and the reduced relation $ b_{1}^{2}=1$.
\begin{equation*} \label{14}
\begin{split}
  b_{i}^{2} &\overset{(*)}=(v_{i-1} \cdots v_{1}b_{1}\underline{v_{1} \cdots v_{i-1})(v_{i-1} \cdots v_{1}}b_{1}v_{1} \cdots v_{i-1}) \\
             &=(v_{i-1} \cdots v_{1}\underline{b_{1}b_{1}}v_{1} \cdots v_{i-1}) \\
             &=\underline{(v_{i-1} \cdots v_{1}v_{1} \cdots v_{i-1})} \\
             &=1. \\
  \end{split}
\end{equation*}
\end{proof}

\subsection{A reduced presentation for the flat twisted braid group}

The set of flat virtual braids on $n$ strands forms a group, the \emph{flat virtual braid group}, denoted
$\mathcal{FV}_n$.
Similarly, the set of flat twisted braids on $n$ strands forms a group, \emph{the flat twisted braid group} denoted $\mathcal{FT}_{n}$, which is generated by flat generators $c_i$, virtual generators $v_{i}$ and twisted generators $b_{i}$ (see Figure \ref{T2}).

\begin{figure}[!htbp]
  \centering
   \subfigure[$c_{i}$]{
  \includegraphics[width=0.28\textwidth]{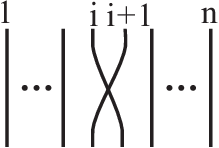}}  ~~~
   \subfigure[$v_{i}$]{
  \includegraphics[width=0.28\textwidth]{C2.eps}}  ~~~
   \subfigure[$b_{i}$]{
  \includegraphics[width=0.28\textwidth]{C3.eps}}\\
  \caption{Generators $c_i$, $v_i$ and $b_i$ of flat twisted braid $\mathcal{FT}_{n}$}\label{T2}
\end{figure}

Among themselves the generators not only satisfy virtual relations, twisted relations and mixed relations that involve $b_{i}$ and $v_{i}$, but also satisfy the following relations:

flat relations:
\begin{equation*}
\begin{split}
&c_{i}^{2}=1, \ \ 1 \leq i \leq n-1,\\
&c_{i}c_{i+1}c_{i}=c_{i+1}c_{i}c_{i+1}, \  \ 1 \leq i \leq n-2,\\
&c_{i}c_{j}=c_{j}c_{i}, \   \ |i-j|\geq 2, \\
 \end{split}
\end{equation*}

mixed flat relations:
\begin{equation*}
\begin{split}
 &c_{i}v_j=v_jc_{i}, \  \ |i-j|\geq 2,\\
 &v_{i}c_{i+1}v_i=v_{i+1}c_{i}v_{i+1},  \ \ 1 \leq i \leq n-2,\\
 &b_ib_{i+1}c_{i}b_{i+1}b_i=v_{i}c_{i}v_i, \  \ 1 \leq i \leq n-1,\\
 &b_{i}c_{j}=c_{j}b_{i}, \  \ j>i \ or \ j<i-1. \\
 \end{split}
\end{equation*}

As for the twisted braids, we have for the flat crossings the inductive defining relations $c_{i+1}=v_{i}v_{i+1}c_{i}v_{i+1}v_{i}$, which leads to the defining relations:
$$ b_{i+1}:=(v_{i}v_{i-1}\cdots v_1)b_1(v_1v_{2}\cdots v_{i}),\ \ i=1, \cdots, n-1, \eqno{(*)}$$
$$ c_{i+1}:=(v_{i} \cdots v_{2}v_{1})(v_{i+1}\cdots v_{3}v_{2})c_{1}(v_{2}v_{3}\cdots v_{i+1})(v_{1} v_{2}\cdots v_{i}), \ \ i=1, \cdots, n-2. \eqno{(\S)}$$

For the \emph{flat virtual braid group} $\mathcal{FV}_n$, Kauffman and Lambropoulou \cite{Kauffman2004} gave following reduced presentation for $\mathcal{FV}_{n}$.

\begin{thm}(\cite{Kauffman2004}, Theorem 3)\label{AFVB}
The flat virtual braid group $\mathcal{FV}_{n}$ has the following reduced
presentation:
\begin{equation*}\label{1}
\begin{split}
  \mathcal{FV}_{n}=\langle \ c_{1}, v_{1}, \cdots, v_{n-1} \ | \ &v_{i}v_{i+1}v_{i}=v_{i+1}v_{i}v_{i+1}, \ v_{i}v_{j}=v_{j}v_{i}, \ |i-j|\geq 2 ,\\
  &v_{i}^{2}=1, \ 1 \leq i \leq n-1, \ c_{1}v_{j}=v_{j}c_{1}, \ j>2, \  c_{1}^{2}=1, \\
  &(v_{1}c_{1}v_{1})(v_{2}c_{1}v_{2})(v_{1}c_{1}v_{1})=(v_{2}c_{1}v_{2})(v_{1}c_{1}v_{1})(v_{2}c_{1}v_{2}), \\
 &c_{1}(v_{2}v_{3}v_{1}v_{2}c_{1}v_{2}v_{1}v_{3}v_{2})=(v_{2}v_{3}v_{1}v_{2}c_{1}v_{2}v_{1}v_{3}v_{2})c_{1}  \ \rangle.
  \end{split}
\end{equation*}
\end{thm}

For the flat twisted braid group $\mathcal{FT}_{n}$, we have:

\begin{thm}
The flat twisted braid group $\mathcal{FT}_{n}$ has the following reduced
presentation:
\begin{equation*}\label{vrelation}
\begin{split}
  \mathcal{FT}_{n}=\langle \ c_{1}, b_{1}, v_{1}, \cdots, v_{n-1} \ | \  &v_{i}v_{i+1}v_{i}=v_{i+1}v_{i}v_{i+1}, \ v_{i}v_{j}=v_{j}v_{i}, \  |i-j|\geq 2 ,\\
  &v_{i}^{2}=1, \ 1 \leq i \leq n-1, \ c_{1}v_{j}=v_{j}c_{1}, \  j>2 ,\\
  &c_{1}^{2}=1, \ b_{1}^{2}=1, \ b_{1}v_{j}=v_{j}b_{1}, \ j>1\  \\
&(v_{1}c_{1}v_{1})(v_{2}c_{1}v_{2})(v_{1}c_{1}v_{1})=(v_{2}c_{1}v_{2})(v_{1}c_{1}v_{1})(v_{2}c_{1}v_{2}),\\
 &c_{1}(v_{2}v_{3}v_{1}v_{2}c_{1}v_{2}v_{1}v_{3}v_{2})=(v_{2}v_{3}v_{1}v_{2}c_{1}v_{2}v_{1}v_{3}v_{2})c_{1},\\
  & b_{1}(v_{1}b_{1}v_{1})=(v_{1}b_{1}v_{1})b_{1},\\
  &c_{1}(v_{2}v_{1}b_{1}v_{1}v_{2})=(v_{2}v_{1}b_{1}v_{1}v_{2})c_{1}, \\
  &(v_{1}b_{1}v_{1})b_1c_{1}b_1 (v_{1}b_{1}v_{1})=v_{1}c_{1}v_1 \ \rangle.\\
  \end{split}
\end{equation*}
\end{thm}

\begin{proof}
{\rm \textcircled{1}} \ The flat relations and mixed flat relations:
\begin{equation*}\label{4}
\begin{split}
  &c_{i}^{2}=1, \ \ 1 \leq i \leq n,\\
  &c_{i}c_{i+1}c_{i}=c_{i+1}c_{i}c_{i+1}, \  \ 1 \leq i \leq n-2, \\
  &c_{i}c_{j}=c_{j}c_{i}, \ \ |i-j|\geq 2,  \\
   &c_{i}v_{j}=v_{j}c_{i}, \ \ |i-j|\geq 2, \\
  &v_ic_{i+1}v_i=v_{i+1}c_iv_{i+1}, \  \ 1 \leq i\leq n-2, \\
  \end{split}
\end{equation*}

follow from the Theorem \ref{AFVB}.

{\rm \textcircled{2}} \ The mixed relations and twisted relations:
\begin{equation*}\label{4}
\begin{split}
  &b_{i}^{2}=1, \ \ 1 \leq i \leq n, \\
  &b_{i}b_{j}=b_{j}b_{i}, \  \ i \neq j,  \\
   &b_{i}v_{j}=v_{j}b_{i}, \  \ j>i \ or \ j<i-1,  \\
  &v_ib_{i+1}=b_iv_{i}, \  \ 1 \leq i \leq n-1, \\
  \end{split}
\end{equation*}

follow from the proof of the Theorem \ref{twisted braid}.

{\rm \textcircled{3}} \ The mixed flat relations $b_{i}c_{j}=c_{j}b_{i}$ for $ j>i$ or $j<i-1$ follow from the defining relations $(*)$ and $(\S)$, the reduced relation $c_{1}(v_{2}v_{1}b_{1}v_{1}v_{2})=(v_{2}v_{1}b_{1}v_{1}v_{2})c_{1}$ and mixed flat relations $c_{i}v_{j}=v_{j}c_{i}$ for $ |i-j|\geq 2$ in $\textcircled{1}$ and $v_{i}b_{i+1}=b_{i}v_{i}$ for $i \geq 1$ and $b_{i}v_{j}=v_{j}b_{i}$ for $j>i$ or $j<i-1$ in $\textcircled{2}$.

If $j > i $, first we have:
\begin{equation*}\label{8}
\begin{split}
  b_{1}c_{j} &\overset{(\S)}=\underline{b_{1}}(v_{j-1} \cdots v_{2}v_{1})(v_{j}\cdots v_{3}v_{2})c_{1}(v_{2}v_{3}\cdots v_{j})(v_{1}v_{2} \cdots v_{j-1}) \\
                  &\overset{\textcircled{2}}=(v_{j-1} \cdots v_{2}\underline{b_{1}v_{1}})(v_{j}\cdots v_{3}v_{2})c_{1}(v_{2}v_{3}\cdots v_{j})(v_{1}v_{2} \cdots v_{j-1}) \\
                  &\overset{\textcircled{2}}=(v_{j-1} \cdots v_{2}v_{1}\underline{b_{2}})(v_{j}\cdots v_{3}v_{2})c_{1}(v_{2}v_{3}\cdots v_{j})(v_{1}v_{2} \cdots v_{j-1}) \\
                   &\overset{\textcircled{2}}=(v_{j-1} \cdots v_{2}v_{1})(v_{j}\cdots v_{3}\underline{b_{2}v_{2}})c_{1}(v_{2}v_{3}\cdots v_{j})(v_{1}v_{2} \cdots v_{j-1})\\
     &\overset{\textcircled{2}}=(v_{j-1} \cdots v_{2}v_{1})(v_{j}\cdots v_{3}
                   v_{2})\underline{b_{3}}c_{1}(v_{2}v_{3}\cdots v_{j})(v_{1}v_{2} \cdots v_{j-1}) \\
                   &\overset{(*)}=(v_{j-1} \cdots v_{2}v_{1})(v_{j}\cdots v_{3}
                   v_{2})\underline{(v_{2}v_{1}b_{1}v_{1}v_{2})c_{1}}(v_{2}v_{3}\cdots v_{j})(v_{1}v_{2} \cdots v_{j-1}) \\
                   &=(v_{j-1} \cdots v_{2}v_{1})(v_{j}\cdots v_{3}
                   v_{2})c_{1}\underline{(v_{2}v_{1}b_{1}v_{1}v_{2})}(v_{2}v_{3}\cdots v_{j})(v_{1}v_{2} \cdots v_{j-1}) \\
                    &\overset{(*)}=(v_{j-1} \cdots v_{2}v_{1})(v_{j}\cdots v_{3}
                   v_{2})c_{1}\underline{b_{3}(v_{2}}v_{3}\cdots v_{j})(v_{1}v_{2} \cdots v_{j-1}) \\
                   &\overset{\textcircled{2}}=(v_{j-1} \cdots v_{2}v_{1})(v_{j}\cdots v_{3}
                   v_{2})c_{1}(v_{2}\underline{b_{2}}v_{3}\cdots v_{j})(v_{1}v_{2} \cdots v_{j-1}) \\
                    &\overset{\textcircled{2}}=(v_{j-1} \cdots v_{2}v_{1})(v_{j}\cdots v_{3}
                   v_{2})c_{1}(v_{2}v_{3}\cdots v_{j})(\underline{b_{2}v_{1}}v_{2} \cdots v_{j-1}) \\
                   &\overset{\textcircled{2}}=(v_{j-1} \cdots v_{2}v_{1})(v_{j}\cdots v_{3}
                   v_{2})c_{1}(v_{2}v_{3}\cdots v_{j})(v_{1}\underline{b_{1}v_{2}} \cdots v_{j-1}) \\
                   &\overset{\textcircled{2}}=\underline{(v_{j-1} \cdots v_{2}v_{1})(v_{j}\cdots v_{3}
                   v_{2})c_{1}(v_{2}v_{3}\cdots v_{j})(v_{1}v_{2}v_{3} \cdots v_{j-1})}b_{1} \\
                   &\overset{(\S)}=c_{j}b_{1}. \\
  \end{split}
\end{equation*}

 Then
\begin{equation*} \label{9}
\begin{split}
  b_{i}c_{j} &\overset{(*)}=(v_{i-1} \cdots v_{1})b_{1}(v_{1} \cdots v_{i-1})\underline{c_{j}} \\
                  &\overset{\textcircled{1}}=c_{j}\underline{(v_{i-1} \cdots v_{1})b_{1}(v_{1} \cdots v_{i-1})} \\
                  &\overset{(*)}=c_{j}b_{i}. \\
  \end{split}
\end{equation*}

If $j < i-1$, first we have:
\begin{equation*} \label{10}
\begin{split}
c_{1} b_{i} &\overset{(*)}=\underline{c_{1}}(v_{i-1} \cdots v_{1})b_{1}(v_{1} \cdots v_{i-1}) \\
                 &\overset{\textcircled{1}}=(v_{i-1} \cdots \underline{c_{1}v_{2}v_{1})b_{1}(v_{1}v_{2}}\cdots v_{i-1}) \\
                 &=(v_{i-1} \cdots v_{2}v_{1})b_{1}(v_{1}v_{2}\underline{c_{1}}\cdots v_{i-1}) \\
                 &\overset{\textcircled{1}}=\underline{(v_{i-1} \cdots v_{2}v_{1})b_{1}(v_{1}v_{2}\cdots v_{i-1})}c_{1} \\
                 &\overset{(*)}=b_{i}c_{1}. \\
  \end{split}
\end{equation*}

Then
\begin{equation*} \label{11}
\begin{split}
  b_{i}c_{j} &\overset{(\S)}=\underline{b_{i}}(v_{j-1} \cdots v_{2}v_{1})(v_{j}\cdots v_{3}v_{2})c_{1}(v_{2}v_{3}\cdots v_{j})(v_{1}v_{2} \cdots v_{j-1}) \\
                  &\overset{\textcircled{2}}=\underline{(v_{j-1} \cdots v_{2}v_{1})(v_{j}\cdots v_{3}v_{2})c_{1}(v_{2}v_{3}\cdots v_{j})(v_{1}v_{2} \cdots v_{j-1})}b_{i} \\
                  &\overset{(\S)}=c_{j}b_{i}. \\
  \end{split}
\end{equation*}

{\rm \textcircled{4}} \ The mixed flat relations $b_{i}b_{i+1}c_{i}b_{i+1}b_{i}=v_{i}c_{i}v_i$ for $i>1$ follow from the above defining relations $(\ddag)$, $(*)$ and $(\S)$, the reduced relations $(v_{1}b_{1}v_{1})b_1c_{1}b_1(v_{1}b_{1}v_{1})=v_{1}c_{1}v_1$ and the mixed flat relations $b_{i}v_{i}=v_{i}b_{i+1}$ and $b_{i}v_{j}=v_{j}b_{i}$ for $ j>i$ or $j<i-1$ in $\textcircled{2}$.
\begin{equation*}\label{2}
\begin{split}
  b_ib_{i+1}c_{i}b_{i+1}b_i&\overset{(*)(\S)}=\underline{b_{i}(v_{i}} \cdots v_{1})b_{1}\underline{(v_{1} \cdots v_{i})(v_{i-1} \cdots v_{1})}(v_{i} \cdots v_{3}v_{2})c_{1}(v_{2}v_{3} \cdots v_{i}) \\& \underline{(v_{1} \cdots v_{i-1})(v_{i}\cdots v_{1})}b_{1}(v_{1} \cdots \underline{v_{i})b_{i}} \\
           &\overset{(\ddag)\textcircled{2}}=(v_{i}\underline{b_{i+1}} \cdots v_{1})b_{1}(v_{i} \cdots v_{2}v_{1}\underline{v_{2} \cdots v_{i})(v_{i} \cdots v_{3}v_{2})}c_{1} \\& \underline{(v_{2}v_{3} \cdots v_{i})(v_{i} \cdots v_{2}}v_{1}v_{2} \cdots v_{i})b_{1}(v_{1} \cdots \underline{b_{i+1}}v_{i}) \\
           &\overset{\textcircled{2}}=(v_{i}\cdots v_{1})b_{1}\underline{b_{i+1} (v_{i}} \cdots v_{2}v_{1})c_{1}(v_{1}v_{2} \cdots \underline{v_{i})b_{i+1}}b_{1}(v_{1} \cdots v_{i}) \\
           &\overset{\textcircled{2}}=(v_{i}\cdots v_{1})b_{1} (v_{i}\underline{b_{i}v_{i-1}} \cdots v_{2}v_{1})c_{1}(v_{1}v_{2} \cdots \underline{v_{i-1}b_{i}}v_{i})b_{1}(v_{1} \cdots v_{i}) \\
           &\overset{\textcircled{2}}= \cdots \\
           &\overset{\textcircled{2}}=(v_{i}\cdots v_{1})\underline{b_{1}}(v_{i}v_{i-1} \cdots v_{2}v_{1})b_{1}c_{1}b_{1}(v_{1}v_{2} \cdots v_{i-1}v_{i})\underline{b_{1}}(v_{1} \cdots v_{i}) \\
           &\overset{\textcircled{2}}=(v_{i}\cdots v_{1})(v_{i}v_{i-1} \cdots v_{2}b_{1}\underline{v_{1})b_{1}}c_{1}\underline{b_{1}(v_{1}}b_{1}v_{2} \cdots v_{i-1}v_{i})(v_{1} \cdots v_{i}) \\
            &\overset{\textcircled{2}}=(v_{i}\cdots v_{1})(v_{i}v_{i-1} \cdots v_{2}b_{1}b_{2}\underline{v_{1})c_{1}(v_{1}}b_{2}b_{1}v_{2} \cdots v_{i-1}v_{i})(v_{1} \cdots v_{i}) \\
             &=(v_{i}\cdots v_{1})(v_{i}v_{i-1} \cdots v_{2}b_{1}b_{2}\underline{(v_{1}b_{1}v_{1})}b_1c_{1}b_1\underline{(v_{1}b_{1}v_{1})}b_{2}b_{1}v_{2} \cdots v_{i-1}v_{i})(v_{1} \cdots v_{i}) \\
            &\overset{(*)}=(v_{i}\cdots v_{1})(v_{i}v_{i-1} \cdots v_{2}\underline{b_{1}b_{2}b_{2}b_1}c_{1}\underline{b_1 b_{2}b_{2}b_{1}}v_{2} \cdots v_{i-1}v_{i})(v_{1} \cdots v_{i}) \\
             &=(v_{i}\underline{v_{i-1}\cdots v_{1})(v_{i}v_{i-1} \cdots v_{2})c_{1}(v_{2}\cdots v_{i-1}v_{i})(v_{1} \cdots v_{i-1}}v_{i}) \\
             &\overset{(\S)}=v_{i}c_{i}v_{i}.\\
  \end{split}
\end{equation*}
\end{proof}

\section{Algebraic Markov Theorem For Twisted Braid}

In \cite{Kauffman2005}, Kauffman and Lambropoulou showed that the following Theorem \ref{Algebraic} is equivalent to the theorem of Kamada's Markov theorem for virtual braids in \cite{Kamada2007}. In this section, we give Algebraic Markov theorem for twisted braid that is the analogue of the Theorem \ref{Algebraic} for virtual braid.

\begin{thm}(Algebraic Markov theorem for virtual braid)(\cite{Kauffman2005}, Theorem 3)\label{Algebraic}
Two oriented virtual links are isotopic if and only if any two corresponding virtual braids differ by a finite sequence of braid relations in $\mathcal{VB}_{\infty}$ and the following moves or their inverses:
\begin{itemize}
  \item [(1)] Virtual and real conjugation: $v_{i} \alpha v_{i} \thicksim  \alpha  \thicksim \sigma_{i}^{-1} \alpha \sigma_{i}$.
  \item [(2)] Right virtual and real stabilization: $\alpha v_{n} \thicksim \alpha \thicksim \alpha \sigma_{n}^{\pm1}$.
  \item [(3)] Algebraic right under-threading: $\alpha \thicksim \alpha \sigma_{n}^{-1} v_{n-1} \sigma_{n}^{+1}$.
  \item [(4)] Algebraic left under-threading: $\alpha \thicksim \alpha v_{n} v_{n-1} \sigma_{n-1}^{+1} v_{n} \sigma_{n-1}^{-1} v_{n-1}v_{n}$.

      where $\alpha, v_{i}, \sigma_{i} \in \mathcal{VB}_{n}$ and $v_{n}, \sigma_{n} \in \mathcal{VB}_{n+1}$ . (see Figure \ref{ALLL})
\end{itemize}
\end{thm}

For twisted braid, we have:

\begin{thm}(Algebraic Markov theorem for twisted braid).\label{TB}
Two oriented twisted links are isotopic if and only if any two corresponding twisted braids differ by a finite sequence of braid relations in $\mathcal{TB}_{\infty}$ and the following moves or their inverses:
\begin{itemize}
  \item [(1)] Virtual, real and twisted conjugation: $v_{i} \alpha v_{i} \thicksim  \alpha  \thicksim \sigma_{i}^{-1} \alpha \sigma_{i} \thicksim b_{i}\alpha b_{i}$.
  \item [(2)] Right virtual and real stabilization: $\alpha v_{n} \thicksim \alpha \thicksim \alpha \sigma_{n}^{\pm1}$.
  \item [(3)] Algebraic right under-threading: $\alpha \thicksim \alpha \sigma_{n}^{-1} v_{n-1} \sigma_{n}^{+1}$.
  \item [(4)] Algebraic left under-threading: $\alpha \thicksim \alpha v_{n} v_{n-1} \sigma_{n-1}^{+1} v_{n} \sigma_{n-1}^{-1} v_{n-1}v_{n}$.

      where $\alpha, v_{i}, \sigma_{i} \in \mathcal{TB}_{n}$ and $v_{n}, \sigma_{n} \in \mathcal{TB}_{n+1}$ (see Figure \ref{ALLL} and Figure \ref{TC}).
\end{itemize}
\end{thm}

\begin{figure}[!htbp]
  \centering
   \subfigure[virtual and real conjugation]{
  \includegraphics[width=0.8\textwidth]{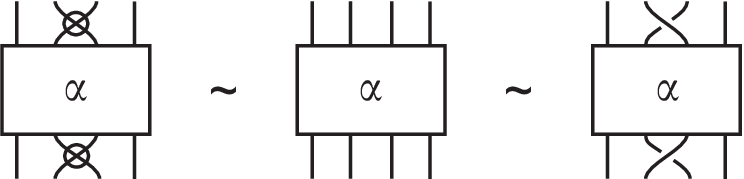}} \
   \subfigure[Right virtual and real stabilizations]{
  \includegraphics[width=0.8\textwidth]{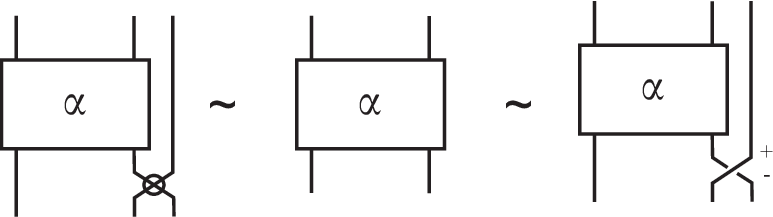}} \
   \subfigure[Algebraic right and left under-threading]{
  \includegraphics[width=0.8\textwidth]{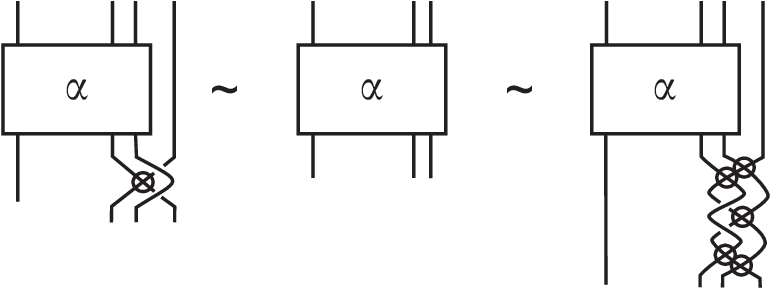}} \
  \caption{The move (1)-(4) of Theorem \ref{Algebraic}}\label{ALLL}\end{figure}

\begin{proof}
The algebraic moves of Theorem \ref{TB} follow immediately from Theorem \ref{Algebraic} and the proof of Theorem \ref{UU}.
\end{proof}

\begin{figure}[!htbp]
  \centering
  \includegraphics[width=0.5\textwidth]{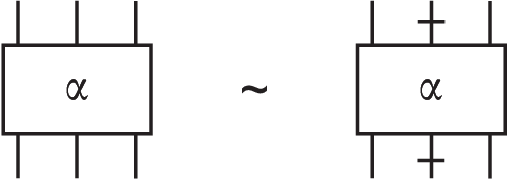}\\
  \caption{Twisted conjugation}\label{TC}
\end{figure}

For flat virtual braid, Kauffman and Lambropoulou \cite{Kauffman2005} had the following result.

\begin{thm}(Algebraic Markov theorem for flat virtuals)(\cite{Kauffman2005}, Theorem 5) \label{AVB}
Two oriented flat virtual links are isotopic if and only if any two corresponding flat virtual braids differ by a finite sequence of braid relations in $\mathcal{FV}_{\infty}$ and the following moves or their inverses:
\begin{itemize}
  \item [(1)] Virtual and flat conjugation: $v_{i} \alpha v_{i} \thicksim  \alpha  \thicksim c_{i} \alpha c_{i}$.
  \item [(2)] Right virtual and flat stabilization: $\alpha v_{n} \thicksim \alpha \thicksim \alpha c_{n}$.
  \item [(3)] Algebraic right flat threading: $\alpha \thicksim \alpha c_{n} v_{n-1} c_{n}$.
  \item [(4)] Algebraic left flat threading: $\alpha \thicksim \alpha v_{n} v_{n-1} c_{n-1} v_{n} c_{n-1} v_{n-1}v_{n}$.

      where $\alpha, v_{i}, c_{i} \in \mathcal{FV}_{n}$ and $v_{n}, c_{n} \in \mathcal{FV}_{n+1}$ (see Figure \ref{ALLL}, substituting the real crossings by the flat ones ).
\end{itemize}
\end{thm}

For flat twisted braid, we have:

\begin{thm}(Algebraic Markov theorem for flat twisted braid). \label{ATB}
Two oriented flat twisted links are isotopic if and only if any two corresponding flat twisted braids differ by a finite sequence of braid relations in $\mathcal{FT}_{\infty}$ and the following moves or their inverses:
\begin{itemize}
  \item [(1)] Virtual, flat and twisted conjugation: $v_{i} \alpha v_{i} \thicksim  \alpha  \thicksim c_{i} \alpha c_{i} \thicksim b_{i}\alpha b_{i}$.
  \item [(2)] Right virtual and flat stabilization: $\alpha v_{n} \thicksim \alpha \thicksim \alpha c_{n}$.
  \item [(3)] Algebraic right flat threading: $\alpha \thicksim \alpha c_{n} v_{n-1} c_{n}$.
  \item [(4)] Algebraic left flat threading: $\alpha \thicksim \alpha v_{n} v_{n-1} c_{n-1} v_{n} c_{n-1} v_{n-1}v_{n}$.

      where $\alpha, v_{i}, c_{i} \in \mathcal{FT}_{n}$ and $v_{n}, c_{n} \in \mathcal{FT}_{n+1}$ (see Figure \ref{ALLL} and Figure \ref{TC}, substituting the real crossings by the flat ones ).
\end{itemize}
\end{thm}

\begin{proof}
The algebraic moves of Theorem \ref{ATB} follow immediately from Theorem \ref{AVB} and the proof of Theorem \ref{UU}.
\end{proof}

\section{Acknowledgements}
The authors would like to express their gratitude to the anonymous referees for their kind suggestions and useful comments on the original manuscript, which resulted in this final version.
Deng is also supported by Doctor's Funds of Xiangtan University (No. 09KZ$|$KZ08069) and NSFC (No. 12001464).
The project is supported partially by Hu Xiang Gao Ceng Ci Ren Cai Ju Jiao Gong Cheng-Chuang Xin Ren Cai (No. 2019RS1057).

\newpage
\section*{Appendix}

\setcounter{figure}{0}
\renewcommand\thefigure{A.\arabic{figure}}

Figure \ref{braiding1} lists the braiding chart of Figure 7 in \cite{Kauffman2005}, And Figure \ref{braiding2}, \ref{braiding3}, \ref{braiding4}, \ref{braiding5}, \ref{braiding6}, \ref{braiding7}, \ref{braiding8}, \ref{braiding9}, \ref{braiding10} list the braiding chart for crossings which contains at least one up-arc in a twisted link diagram.

\begin{figure}[!htbp]
  \centering
  \includegraphics[width=1\textwidth]{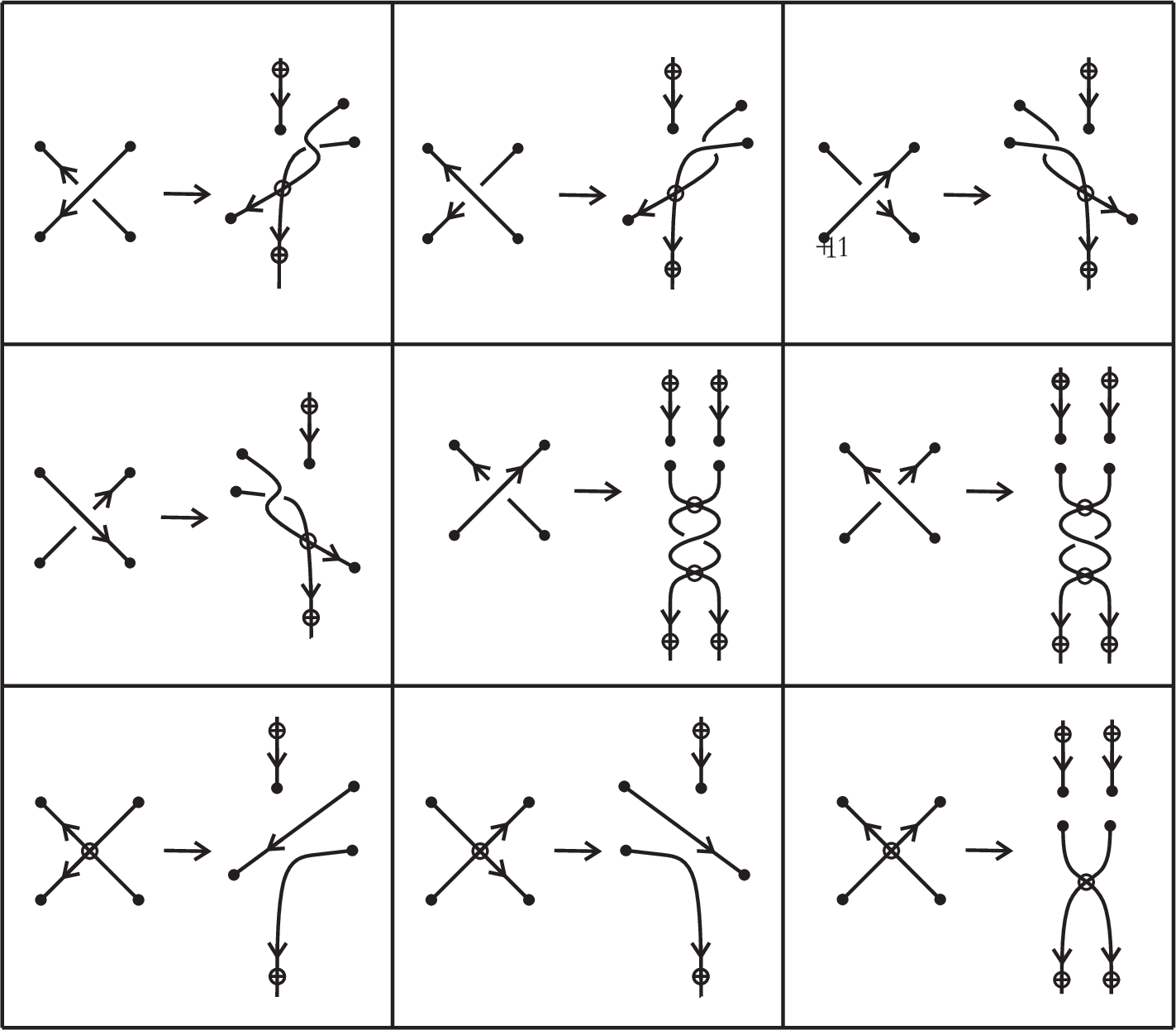}\\
  \caption{}\label{braiding1}
\end{figure}

\begin{figure}[!htbp]
  \centering
  \includegraphics[width=1\textwidth]{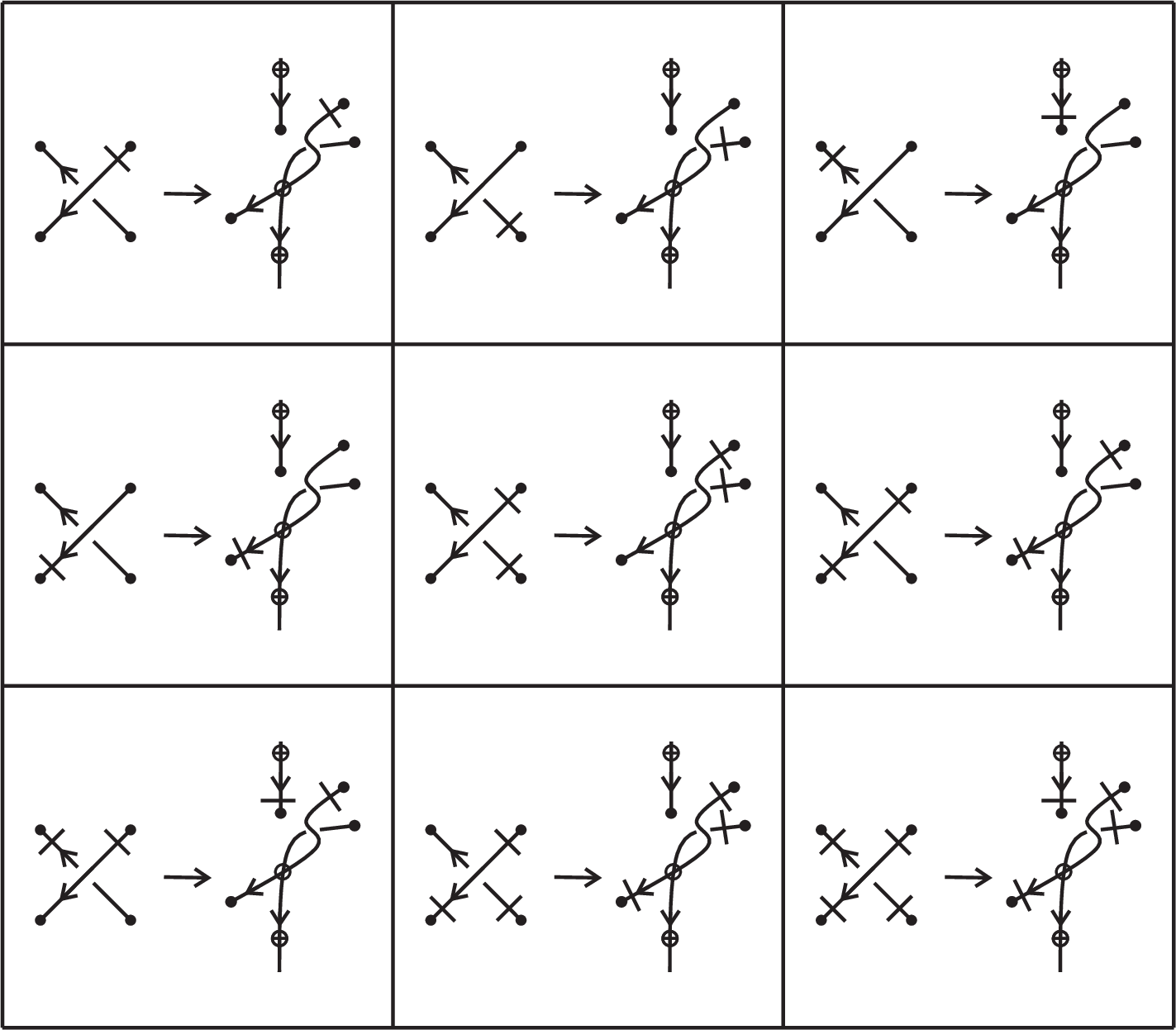}\\
  \caption{}\label{braiding2}
\end{figure}

\begin{figure}[!htbp]
  \centering
  \includegraphics[width=1\textwidth]{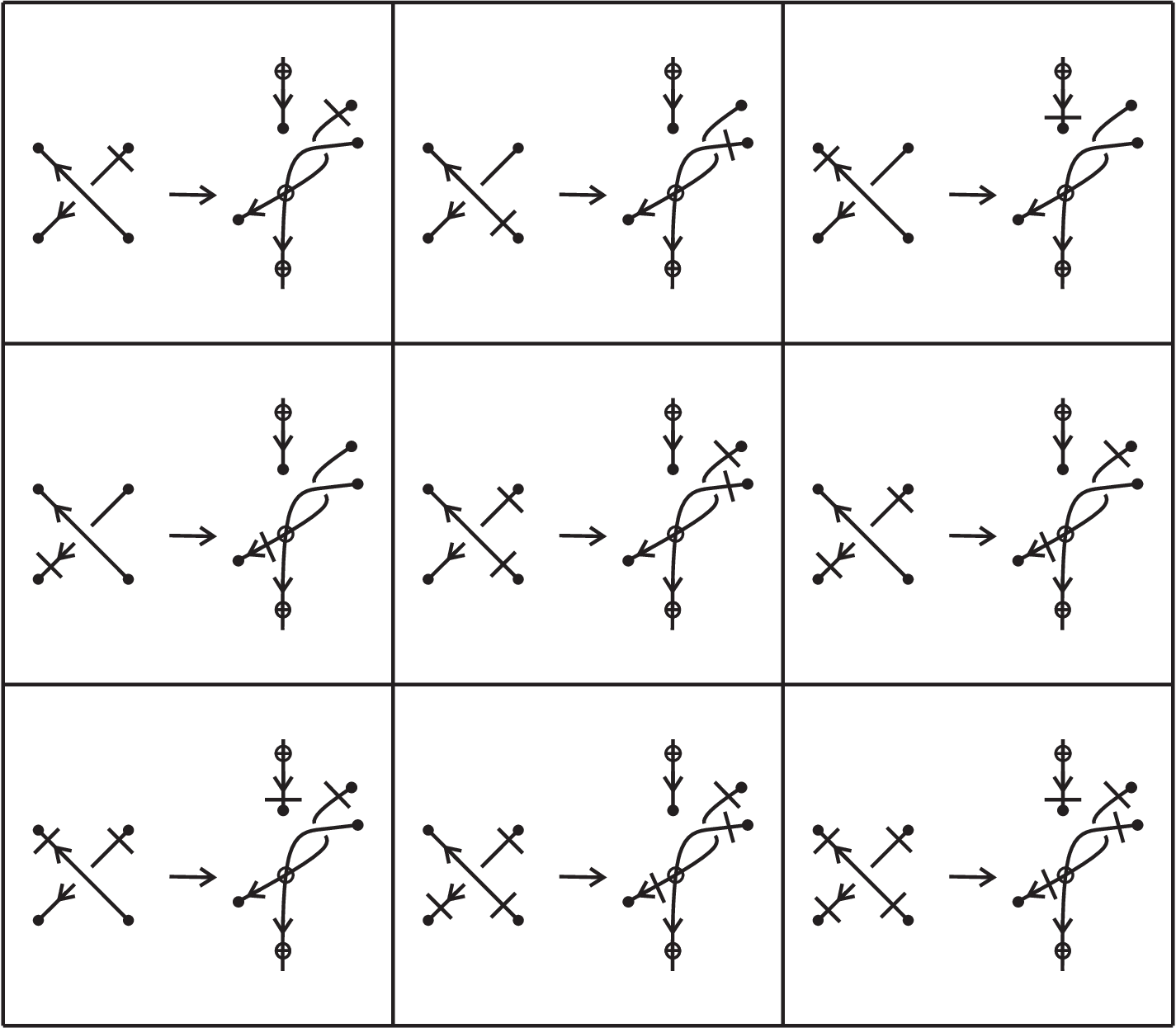}\\
  \caption{}\label{braiding3}
\end{figure}

\begin{figure}[!htbp]
  \centering
  \includegraphics[width=1\textwidth]{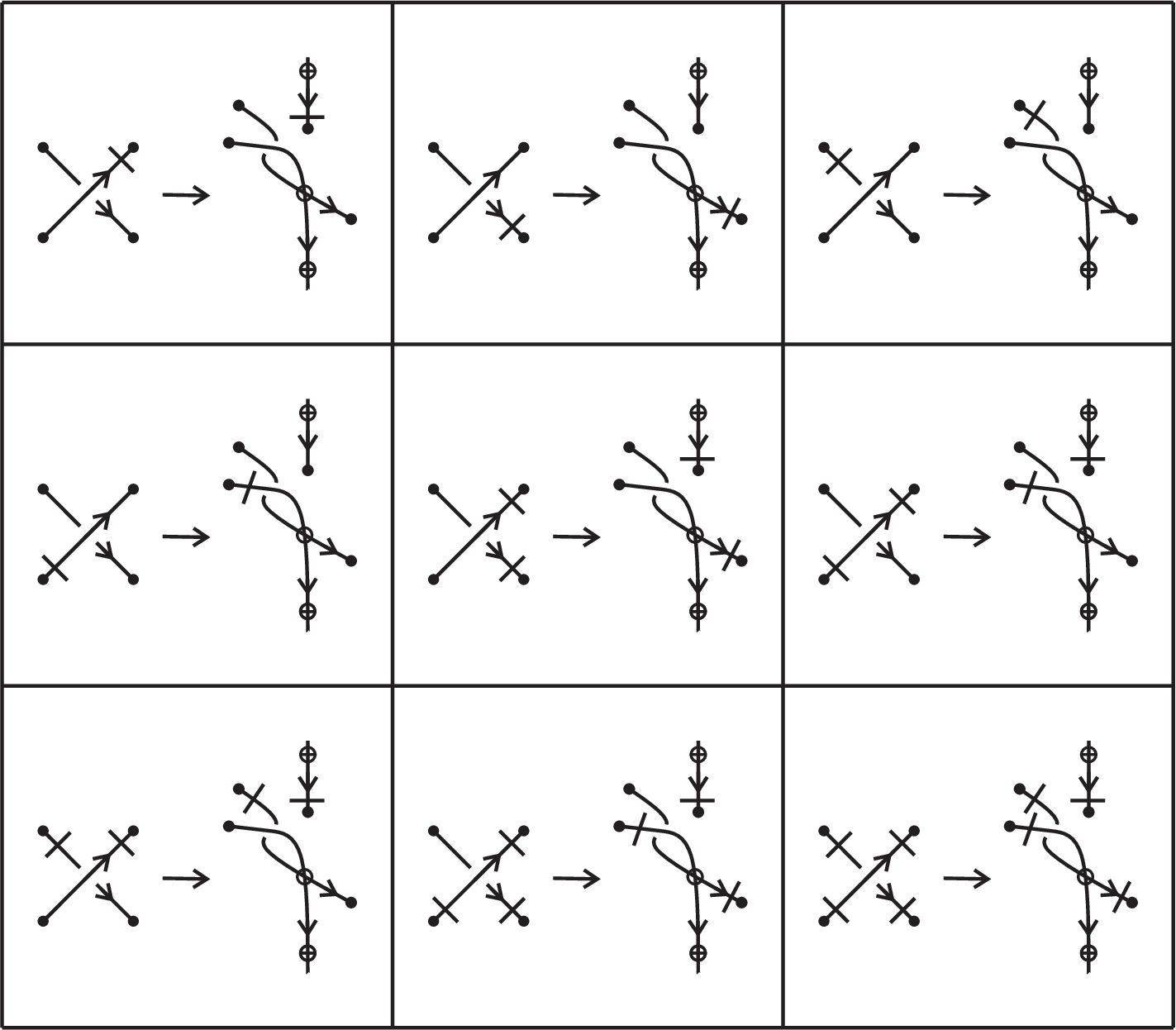}\\
  \caption{}\label{braiding4}
\end{figure}

\begin{figure}[!htbp]
  \centering
  \includegraphics[width=1\textwidth]{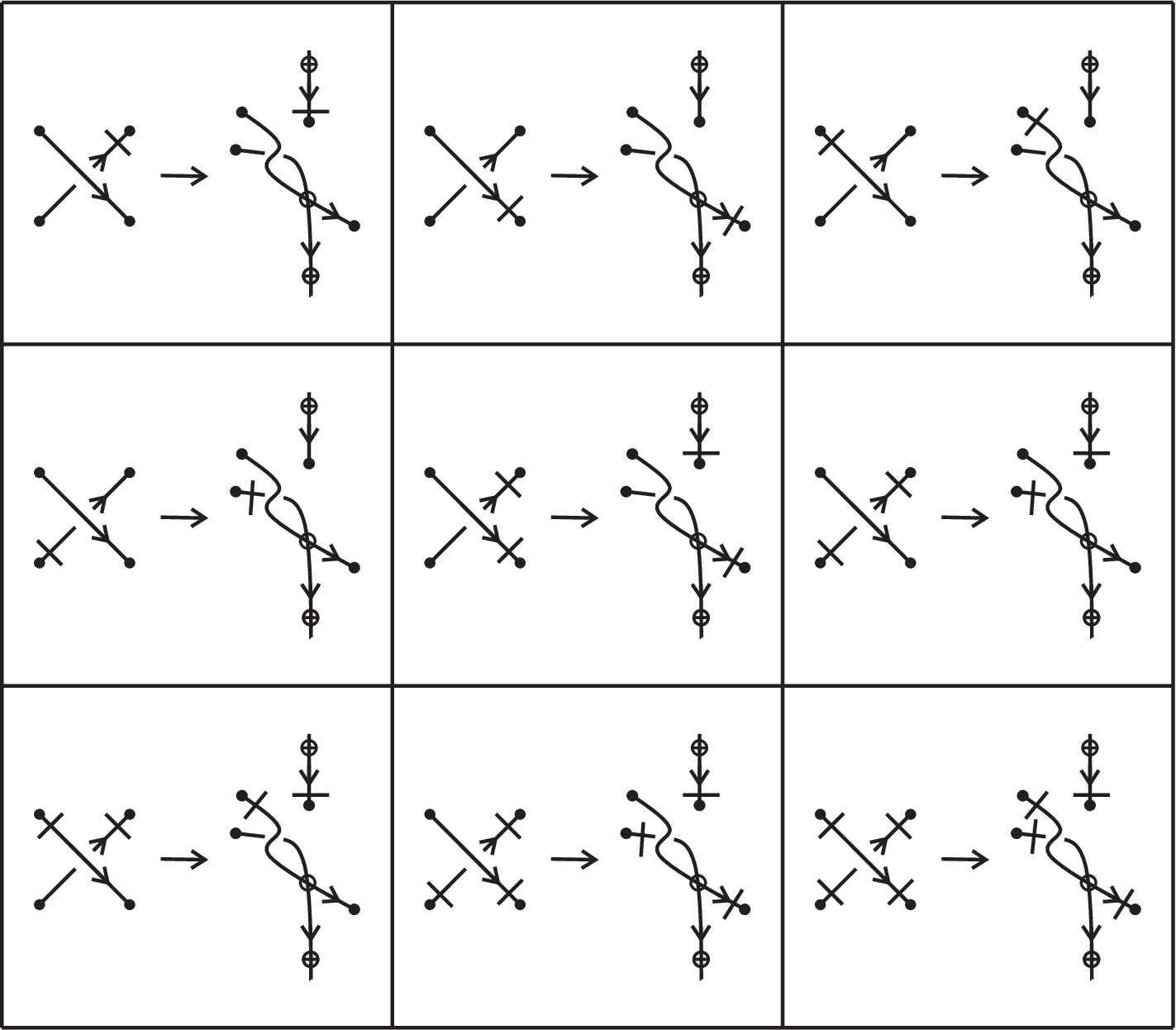}\\
  \caption{}\label{braiding5}
\end{figure}

\begin{figure}[!htbp]
  \centering
  \includegraphics[width=1\textwidth]{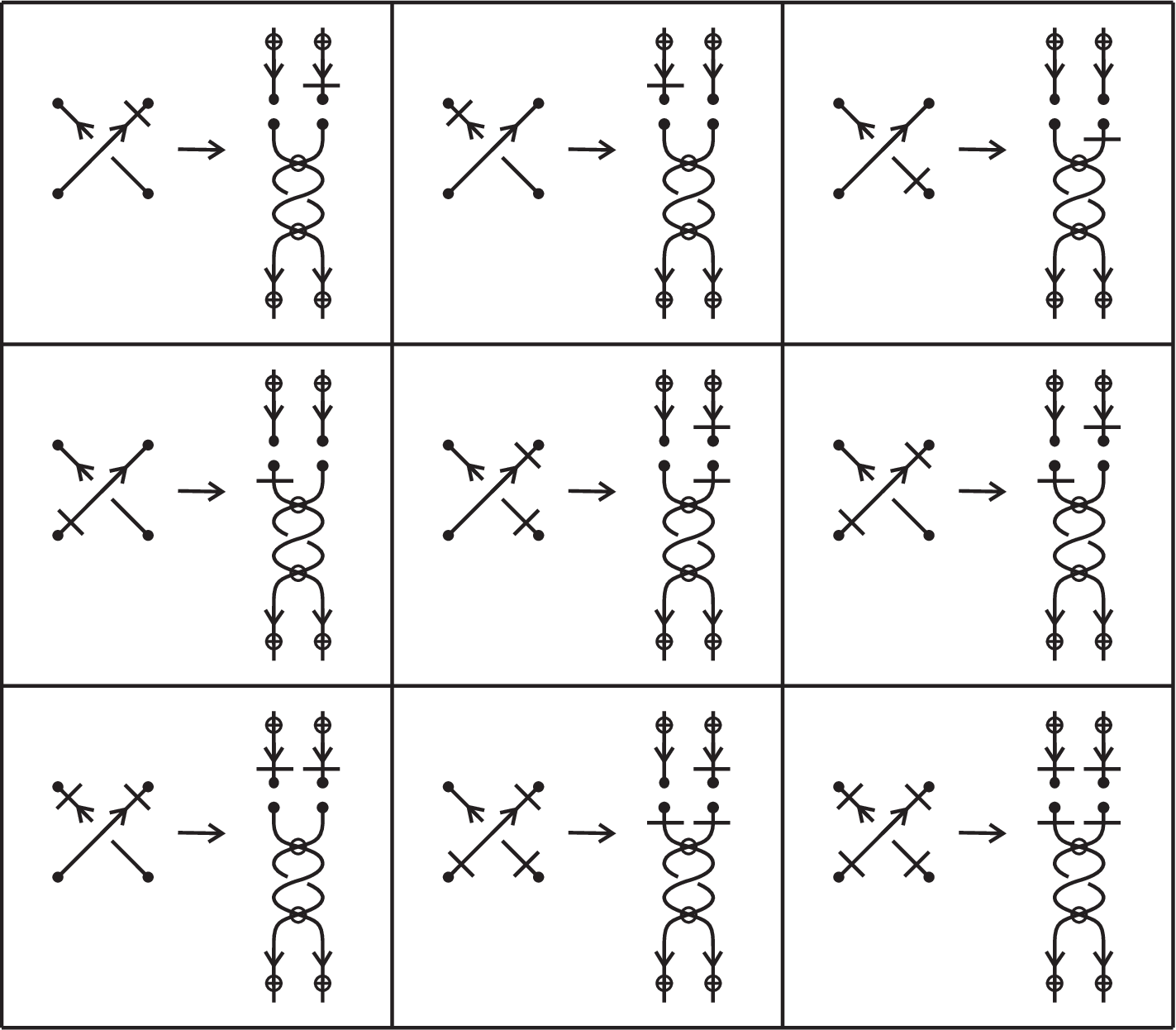}\\
  \caption{}\label{braiding6}
\end{figure}

\begin{figure}[!htbp]
  \centering
  \includegraphics[width=1\textwidth]{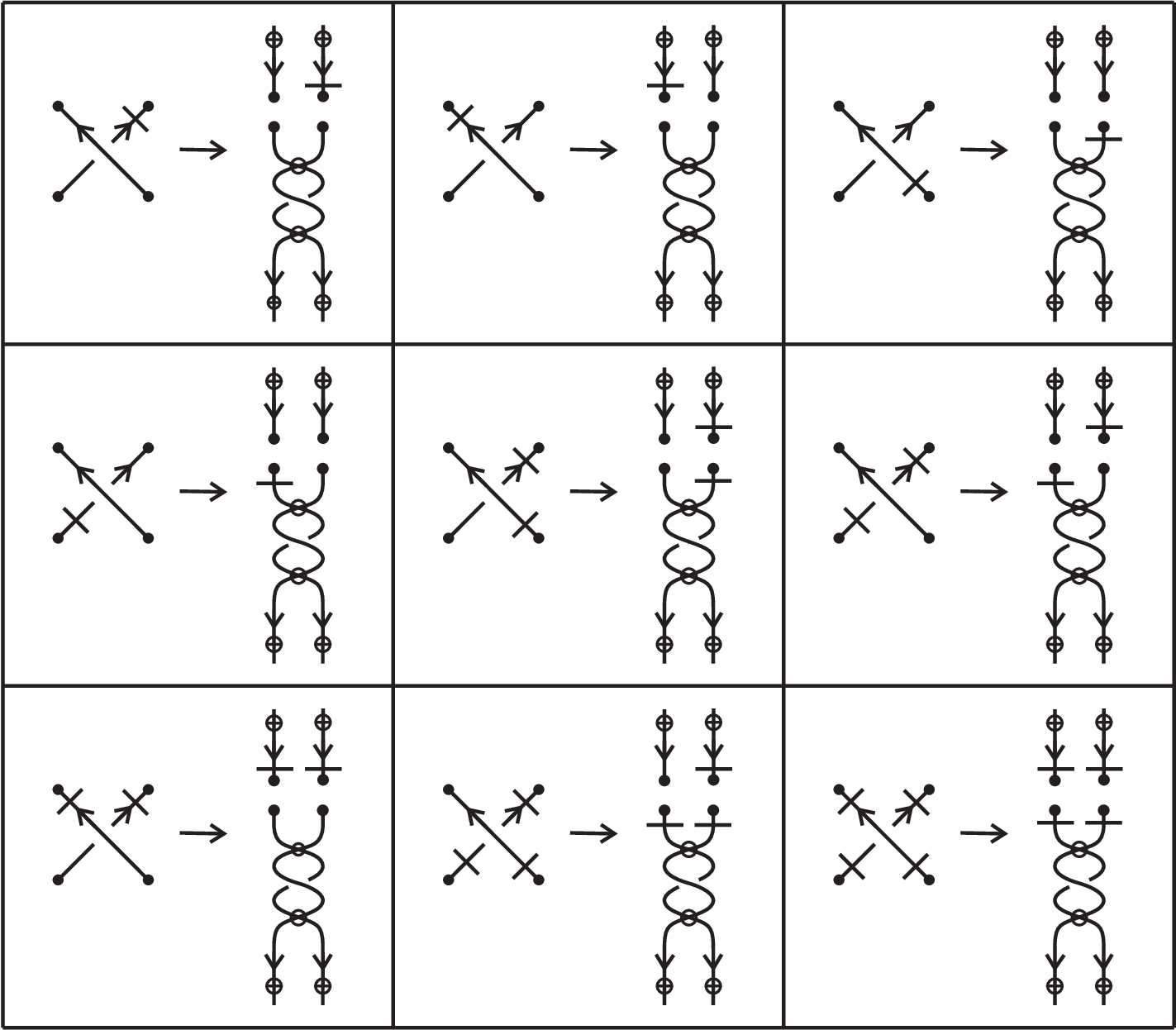}\\
  \caption{}\label{braiding7}
\end{figure}

\begin{figure}[!htbp]
  \centering
  \includegraphics[width=1\textwidth]{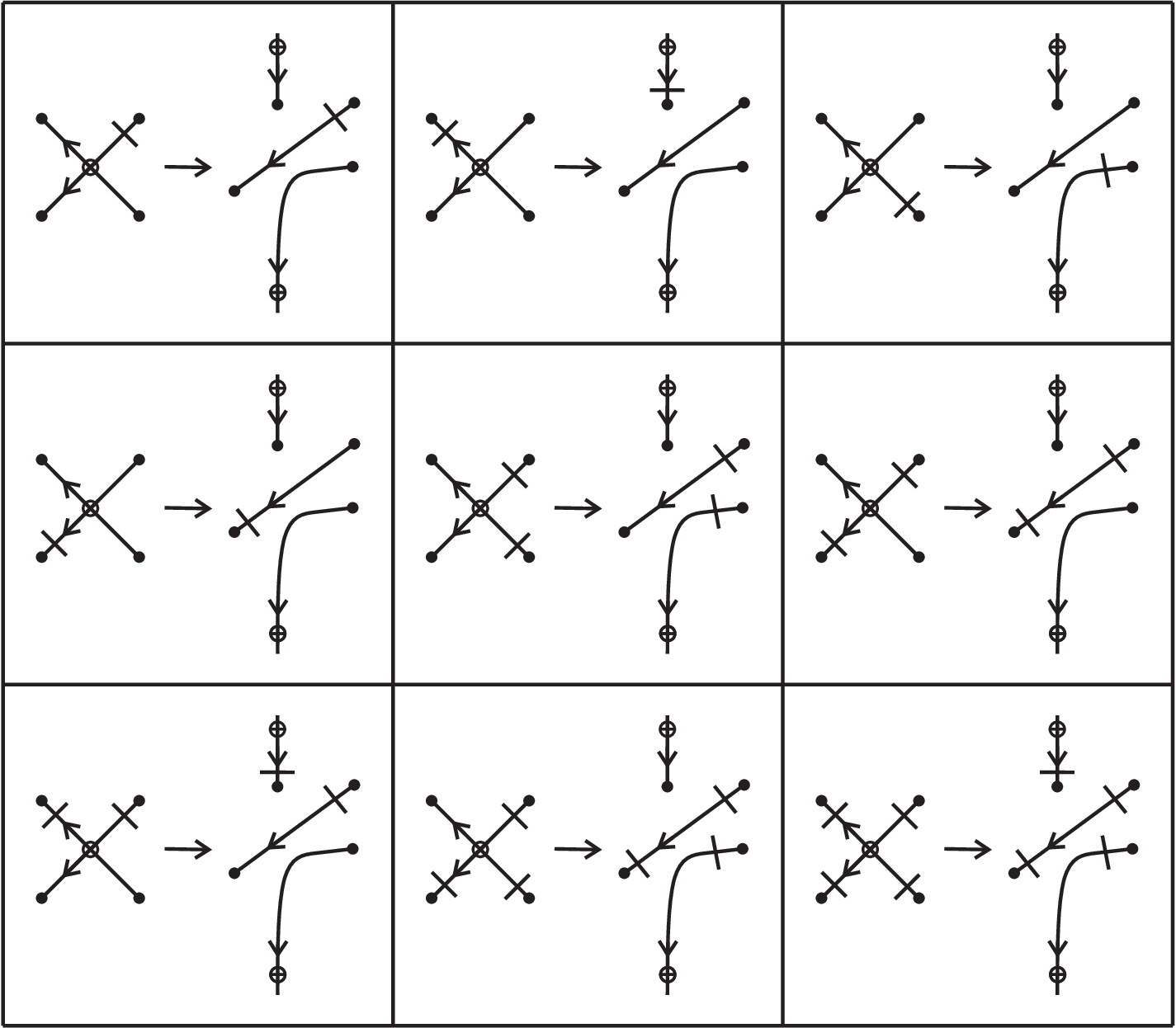}\\
  \caption{}\label{braiding8}
\end{figure}

\begin{figure}[!htbp]
  \centering
  \includegraphics[width=1\textwidth]{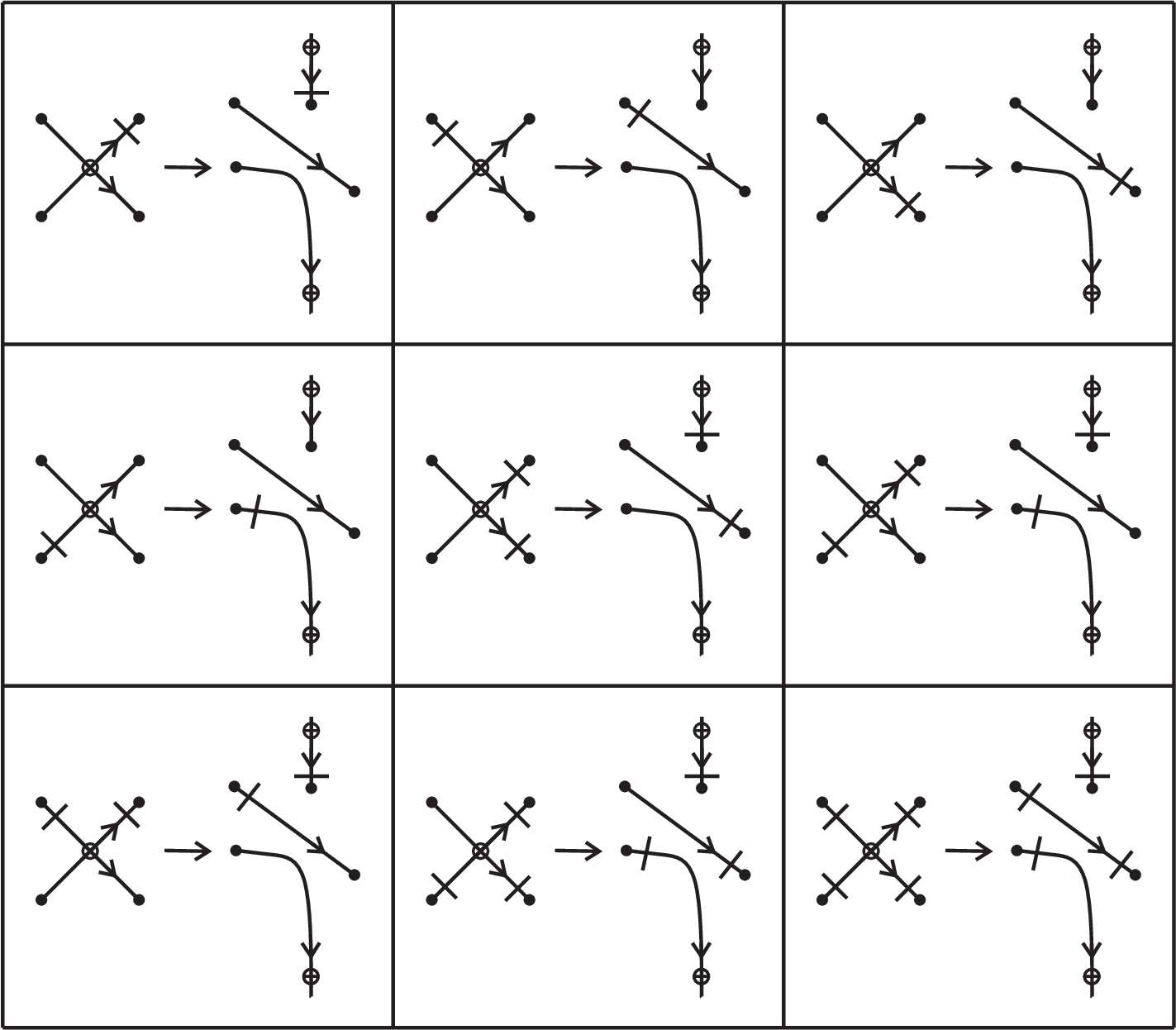}\\
  \caption{}\label{braiding9}
\end{figure}

\begin{figure}[!htbp]
  \centering
  \includegraphics[width=1\textwidth]{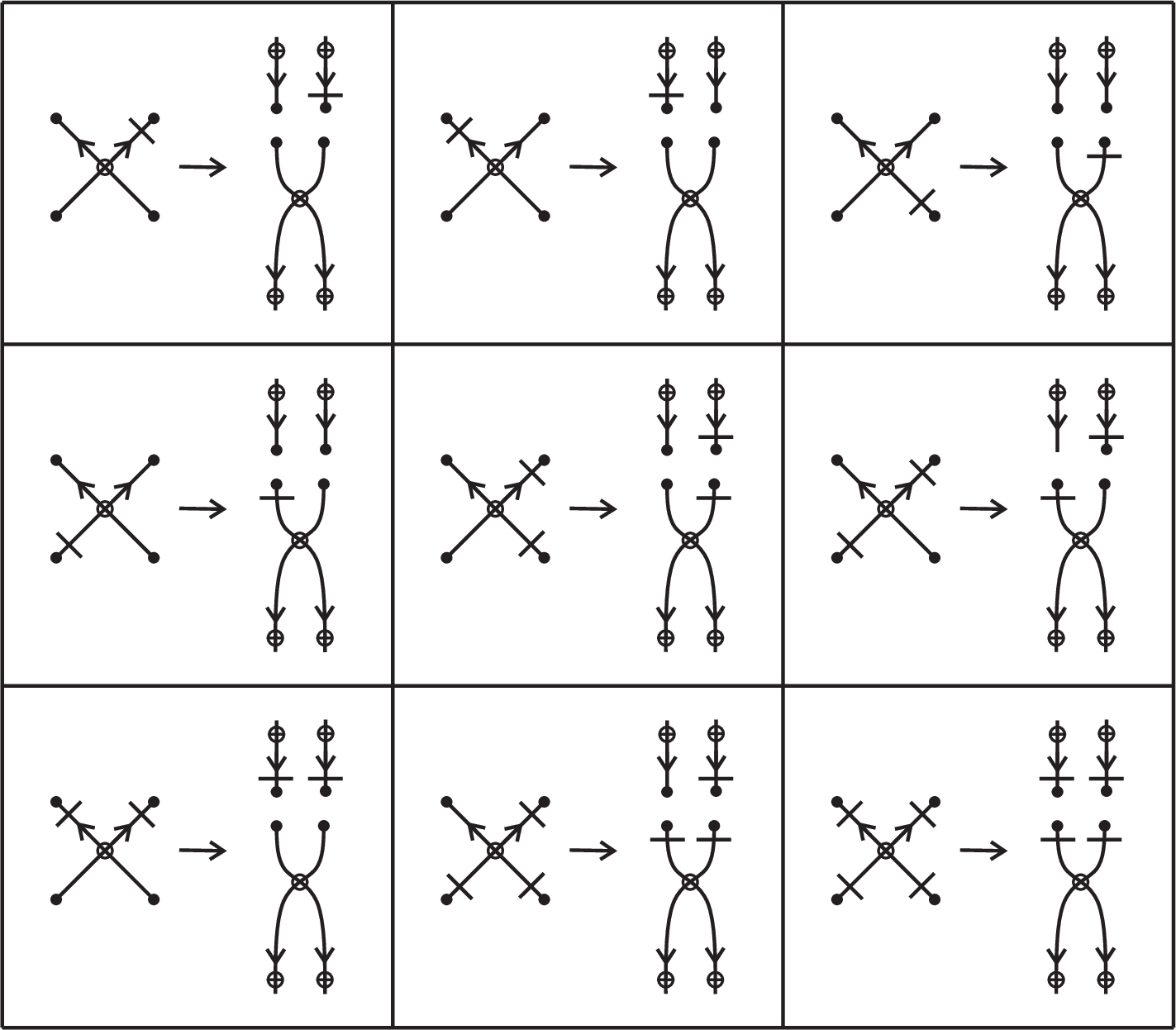}\\
  \caption{}\label{braiding10}
\end{figure}

\end{document}